\documentclass[12 pt]{amsart}
\usepackage{amsmath,amssymb,amsthm}
\usepackage{longtable,}
\usepackage{a4wide}
\usepackage{color}
\usepackage{enumerate}

\usepackage{tikz}
 
\usetikzlibrary{decorations.pathreplacing}

\usepackage{mathrsfs}
\usepackage{chngpage}

\newtheorem{thm}{Theorem}[section]
\newtheorem{bigthm}{Theorem}
\newtheorem*{thm*}{Theorem}
\newtheorem*{conj*}{Conjecture}
\newtheorem{cor}[thm]{Corollary}

\newtheorem{lem}[thm]{Lemma}
\newtheorem{prop}[thm]{Proposition}

\newtheorem{construction}[thm]{Construction}
\newtheorem{hyp}[thm]{Hypothesis}
\theoremstyle{remark}
\newtheorem{rem}{Remark}
\newtheorem*{rem*}{Remark}

\theoremstyle{definition}

\newcounter{claim}[thm]

\newcommand{\PGammaL}{\mathrm{P}\Gamma\mathrm{L}}

\newcommand{\GammaL}{\Gamma\mathrm{L}}

\newcommand{\PGL}{\mathrm{PGL}}

\newcommand{\PSp}{\mathrm{PSp}}
\newcommand{\PO}{\mathrm{PO}}

\newcommand{\Sp}{\mathrm{Sp}}
\newcommand{\PSU}{\mathrm{PSU}}

\newcommand{\POmega}{\mathrm{P}\Omega}

\newcommand{\AGL}{\mathrm{AGL}}

\newcommand{\PSL}{\mathrm{PSL}}

\newcommand{\GL}{\mathrm{GL}}
\newcommand{\SL}{\mathrm{SL}}
\newcommand{\Cos}{\mathrm{Cos}}

\newcommand{\Aut}{\mathrm{Aut}}
\newcommand{\Out}{\mathrm{Out}}

\newcommand{\soc}{\mathrm{soc}}

\title{On edge-primitive $3$-arc-transitive graphs }

\author{Michael Giudici}
\author{Carlisle S.H. King}

\address{
Department of Mathematics and Statistics\\
The University of Western Australia\\
Crawley WA, 6009\\
Australia} 

\email{michael.giudici@uwa.edu.au}	\email{carlisle.king@hotmail.com}

\begin{document}

\maketitle
\begin{abstract}
This paper begins the classification of all edge-primitive 3-arc-transitive graphs by classifying all such graphs where the automorphism group is an almost simple group with socle an alternating or sporadic group, and all such graphs where the automorphism group is an almost simple classical group with a vertex-stabiliser acting faithfully on the set of neighbours.
\end{abstract}

Edge-primitive graphs, that is, graphs whose automorphism group acts primitively on the set of edges, were first studied by Weiss in 1973 \cite{Weiss73} who classified all the edge-primitive graphs of valency three. Many famous graphs are edge-primitive such as the Heawood graph, Tutte-Coxeter graph, Hoffman-Singleton graph and Higman-Sims graph. Another motivation for the study of edge-primitive graphs is the study of graph decompositions \cite{transdecomp}: Given a graph $\Gamma$ and a group $G$ of automorphisms, a partition $\mathcal{P}$ of the set of edges of $\Gamma$ is called a \emph{$G$-transitive decomposition} of $\Gamma$ if $\mathcal{P}$ is $G$-invariant and $G$ acts transitively on $\mathcal{P}$. If $G$ acts transitively on the edge set of $\Gamma$ then $\Gamma$ has no $G$-transitive decompositions if and only if $G$ acts primitively on the set of edges.

The study of edge-primitive graphs was reinvigorated in 2010 by the first author and Li \cite{eprim} by providing a general structure theorem of such graphs and classifying all edge-primitive graphs whose automorphism group contains $\PSL_2(q)$ as a normal subgroup.  This has led to all edge-primitive graphs of valencies 4 \cite{Guo2014b} and 5 \cite{Guo_2012} being classified, and all those of prime valency and having a soluble edge-stabiliser \cite{Pan2018}. Moreover, all edge-primitive graphs of prime power order \cite{PanHW2019} or which are Cayley graphs on abelian and dihedral groups \cite{PanWY2018} have been classified.

For $s\geq 1$, an \emph{$s$-arc} in a graph $\Gamma$ is an $(s+1)$-tuple $(v_0,v_1,\ldots,v_s)$ of vertices such that $v_i\sim v_{i+1}$ but $v_i\neq v_{i+2}$.  We say that $\Gamma$ is \emph{$s$-arc-transitive} if the automorphism group of $\Gamma$ acts transitively on the set of $s$-arcs. If a graph is 1-arc-transitive then we simply refer to it as being \emph{arc-transitive}. If all vertices of $\Gamma$ have valency at least two then an $s$-arc-transitive graph is also $(s-1)$-arc-transitive.  The study of $s$-arc-transitive graphs originated in the seminal work of Tutte \cite{Tutte47, Tutte59}, who showed that a graph of valency three is at most 5-arc-transitive. This was later extended by Weiss \cite{Weiss81} who showed that a graph of valency at least three is at most 7-arc-transitive. The vertex-primitive 4-arc-transitive graphs were classified by Li \cite{Li2001} and all edge-primitive 4-arc-transitive graphs were classified by Li and Zhang \cite{Li2011}. These classifications were enabled by the classification of all vertex-stabiliser, edge-stabiliser pairs for 4-arc-transitive graphs by Weiss \cite{Weiss81}. Moreover, the examples arising  in the edge-primitive case are the point-line incidence graphs of the Desarguesian projective planes, the generalised quadrangles associated with the symplectic groups $\PSp_4(q)$, the generalised hexagons associated with the groups $G_2(q)$ of Lie type, and five other sporadic examples.   A key part of the classification in the edge-primitive case is that the edge-stabiliser is always soluble. 
 Han, Liao and Lu \cite{Han2019} have subsequently classified all edge-primitive graphs for almost simple groups with soluble edge-stabilisers.

We say that an edge-primitive graph is \emph{nontrivial} if it is connected, arc-transitive and has valency at least 3. Let $\Gamma$ be a nontrivial edge-primitive graph and let $G=\Aut(\Gamma)$.  For a vertex $v$ of $\Gamma$, denote the stabiliser of $v$ in $G$ by $G_v$. Suppose that $\{u,v\}$ is an edge. Then the edge-stabiliser $G_{\{u,v\}}$ is maximal in $G$. Moreover, the arc-stabiliser $G_{uv}$ is an index two subgroup of $G_{\{u,v\}}$ and also contained in two other subgroups, namely the vertex-stabilisers $G_v$ and $G_w$.  Suppose further that $\Gamma$ is 2-arc-transitive.  Then  Lu \cite{Lu2018} has shown  that if $\Gamma$ is not complete bipartite then $G$ is almost simple, that is, $G$ has a unique minimal normal subgroup $T$ and $T$ is a nonabelian simple group. We refer to $T$ as the \emph{socle} of $G$ and denote it by $\soc(G)$. He further showed that if  $\Gamma$ is 3-arc-transitive then either the graph has valency 7 and $G_v=\mathrm{A}_7$ or $\mathrm{S}_7$, or $G_v$ acts unfaithfully on the set $\Gamma(v)$ of neighbours of $v$. These observations make a classification of edge-primitive 3-arc-transitive graphs feasible, whereas it would appear that we are far from a classification of all vertex-primitive 3-arc-transitive graphs.  This paper is the first in a series aiming to classify all edge-primitive 3-arc-transitive graphs. 

Let $G$ be a group with core-free subgroup $H$, that is $\cap_{x\in G}H^x = 1$,  and let $g\in G$ be such that $g^2\in H$ and $g$ does not normalise $H$. Then we can construct the coset graph $\Gamma=\Cos(G,H,HgH)$ whose vertices are the right cosets of $H$ in $G$ and $Hx\sim Hy$ if and only if $xy^{-1}\in HgH$. Then $G$ acts faithfully as an arc-transitive group of automorphisms of $\Gamma$. Indeed all arc-transitive graphs arise in this manner by taking $G$ to be a group of automorphisms that acts transitively on the set of arcs, $H$ to be the stabiliser in $G$ of a vertex $v$ and $g\in G$ to be an element interchanging the two vertices of an edge $\{v,u\}$.  See for example \cite{Sabidussi}. We can use the coset graph construction to determine precisely when for a group $G$ with maximal subgroup $E$, there is an edge-primitive graph with group of automorphisms $G$ and edge-stabiliser $E$, see Lemma  \ref{prop:cosetgraph}.

Our first result deals with almost simple groups whose socle is either an alternating or a sporadic simple group. Our group theory notation in Table \ref{tab:Anspor} follows Atlas \cite{Atlas} notation, which will be defined in Section \ref{sec:prelim}.

\begin{bigthm}\label{thm:Anspor}
Let $\Gamma$ be a nontrivial edge-primitive 3-arc-transitive graph with $G=\Aut(\Gamma)$ such that $G$ is an almost simple group whose socle is either an alternating group or a sporadic simple group. Then there is a quadruple $(G,E,A,H)$ as listed in Table \ref{tab:Anspor} such that $E$ is a maximal subgroup of $G$, $A=H\cap E$ and $\Gamma=\Cos(G,H,HgH)$ for some $g\in H\backslash A$.
\end{bigthm}

\begin{table}[!h]
\caption{Edge-primitive 3-arc-transitive graphs with $\Gamma$ and $G$ as in Theorem \ref{thm:Anspor}.}
\label{tab:Anspor}
\centering
\begin{tabular}{c c c c c c c} 
$G$ & $E$ & $A$ & $H$  & Notes \\ \hline \\[-2.5ex]
$\Aut(\mathrm{A}_6)$ & $[2^5]$ & $[2^4]$ & $\mathrm{S}_4 \times \mathrm{S}_2$  & Tutte's 8-cage \cite{Tutte47}, bipartite \\
$\mathrm{M}_{12}.2$ &  $3^{1+2}_+:\mathrm{D}_8$ & $3^{1+2}_+:2^2$ & $3^2:2\mathrm{S}_4$  & Weiss \cite{Weiss85}, bipartite \\
$\mathrm{J}_3.2$ &  $[2^6]:(\mathrm{S}_3)^2$ & $[2^6]:\left( (\mathrm{S}_3)^2 \cap \mathrm{A}_6 \right)$& $[2^4]:(3 \times \mathrm{A}_5).2$  & Weiss \cite{Weiss86}\\
$\mathrm{Ru}$ & $5^{1+2}_+:[2^5]$ & $5^{1+2}_+:[2^4]$ &  $5^2:\GL_2(5)$  & Ru graph \cite{StrWe} \\
$\mathrm{O'N}.2$ & $\PGL_2(9)$ & $\mathrm{A}_6$ & $\mathrm{A}_7$  &  Lu \cite{Lu2018}, bipartite  \\
 \hline
\end{tabular}
\end{table}

Our second result looks at classical groups where the stabiliser of a vertex $v$ acts faithfully on the set of neighbours of $v$. Classical groups with an unfaithful vertex-stabiliser will be dealt with in a subsequent paper. 

\begin{bigthm}\label{thm:faithful}
Let $\Gamma$ be a nontrivial edge-primitive 3-arc-transitive graph  with $G=\Aut(\Gamma)$ such that $G$ is an almost simple classical group and  for a vertex $v$, the vertex-stabiliser $G_v$ acts faithfully on the set of neighbours of $v$. Then there is a quadruple $(G,E,A,H)$ as listed in Table \ref{tab:classicalgroupsfaithful} such that $E$ is a maximal subgroup of $G$, $A=H\cap E$ and $\Gamma=\Cos(G,H,HgH)$ for some $g\in H\backslash A$.
\end{bigthm}

\begin{rem}
The notation for the classical groups is defined in Section \ref{sec:classicalgroupsfaithful}, and for the outer automorphisms $\delta, \delta', \phi$ and $\gamma$ we follow the notation of \cite{BHRD}.  Moreover, $(2^2)_{122}$ is the outer automorphism group of $\PSU_4(3)$ isomorphic to $C_2^2$ given in \cite{Atlas}. The graph obtained from the first row of Table \ref{tab:classicalgroupsfaithful} is the Hoffman-Singleton graph \cite{HoSi} and is the only vertex-primitive graph obtained.  We believe that all of the remaining examples are new.  
\end{rem}

\begin{rem}
Table \ref{tab:classicalgroupsfaithful} also lists a subgroup $S$ of $G$ for each graph. If $G$ acts primitively on vertices then $S=G$ while if $\Gamma$ is bipartite and the stabiliser $G^+$ in $G$ of each bipartite half acts primitively on each bipartite half then $S=G^+$. In all other cases, $S$ is a maximal subgroup of $G$ or $G^+$ that contains $A$.
\end{rem}

\begin{table}[!h]
\begin{adjustwidth}{-1.2in}{-1.2in} 
\caption{Edge-primitive 3-arc-transitive graphs with $\Gamma$ and $G$ as in Theorem \ref{thm:faithful} and $T=\soc(G)$, where $\delta', \delta, \phi$ and $\gamma$ follow the notation of \cite{BHRD}}
\label{tab:classicalgroupsfaithful}
\centering
\begin{tabular}{c c c c c c c c} 
$G$ & $E$ & $A$ & $H$ & $S$ & Conditions \\ \hline \\[-2.5ex]
$\PSU_3(5).\langle \gamma \rangle$ & $\Aut(\mathrm{A_6})$ & $\mathrm{S_6}$ & $\mathrm{S_7}$ & $G$ & \\
$\PSL_6(q).\langle \gamma, \delta^3 \rangle$ & $\Aut(\mathrm{A}_6)$ & $\mathrm{S_6}$ & $\mathrm{S_7}$ & $\PSU_4(3).(2^2)_{122}$ & $q=p \equiv 7, 13\pmod{24}$    \\
$\PSL_6(q).\langle \gamma \rangle$ & $\PGL_2(9)$ & $\mathrm{A_6}$ & $\mathrm{A_7}$ & $T$ & $q=p \equiv 1, 31 \pmod{48}$   \\
$\PSL_6(q).\langle \gamma\delta \rangle$ & $\PGL_2(9)$ & $\mathrm{A_6}$ & $\mathrm{A_7}$ & $T$ & $q=p \equiv 7, 25 \pmod{48}$  \\
$\PSL_6(q).\langle \phi, \gamma \delta^3 \rangle$ & $\Aut(\mathrm{A_6})$ & $\mathrm{S_6}$ & $\mathrm{S_7}$ & $ T. \langle \phi \rangle$ & $q=p^2,$   $p \equiv 5, 11 \pmod{24}$  \\ 
$\PSL_6(q).\langle \phi \gamma, \gamma \delta^3 \rangle$ & $\Aut(\mathrm{A_6})$ & $\mathrm{S_6}$ & $\mathrm{S_7}$ & $T. \langle \phi \gamma \rangle$ & $q=p^2,$  $p \equiv 13, 19 \pmod{24}$  \\ 
$\PSU_6(q).\langle \gamma, \delta^3 \rangle$ & $\Aut(\mathrm{A}_6)$ & $\mathrm{S_6}$ & $\mathrm{S_7}$ & $\PSU_4(3).(2^2)_{122}$ & $q=p \equiv 11,17 \pmod{24}$    \\
$\PSU_6(q).\langle \gamma \rangle$ & $\PGL_2(9)$ & $\mathrm{A_6}$ & $\mathrm{A_7}$ & $T$ & $q=p \equiv 17, 47 \pmod{48}$  \\ 
$\PSU_6(q).\langle \gamma \delta \rangle$ & $\PGL_2(9)$ & $\mathrm{A_6}$ & $\mathrm{A_7}$ & $T$ &  $q=p \equiv 23, 41 \pmod{48}$   \\ 
$\POmega_{10}^-(7).\langle \gamma, \delta' \rangle$ & $\Aut(\mathrm{A_6})$ & $\mathrm{S}_6$ & $\mathrm{S}_7$ & $\mathrm{M}_{22}.2, \PSp_4(7).2$ \\
 \hline
\end{tabular}
\end{adjustwidth}
\end{table}

 We briefly outline the structure of the paper. First we list some preliminary results on edge-primitive graphs and outline our method for proving Theorems \ref{thm:Anspor} and \ref{thm:faithful}. The method involves considering each family of finite almost simple group $G$ of alternating, sporadic or classical type and each type of maximal subgroup $E \leqslant G$. In Sections 2 and 3 we consider the case where $G$ is alternating and sporadic respectively. We consider each type of maximal subgroup of $G$ using the O'Nan-Scott theorem if $G$ is alternating and the Atlas \cite{Atlas} if $G$ is sporadic. We conclude Section 3 with the proof of Theorem \ref{thm:Anspor}. In Section 4 we consider the case where $G$ is a finite almost simple classical group: we make the additional assumption that the local action is faithful, and end with the proof of Theorem \ref{thm:faithful}.

 We note that our analysis unveils some new edge-primitive 2-arc-transitive graphs: one with $G=\mathrm{S}_8$ (Construction \ref{con:gamma1}) and one with $G=\mathrm{J}_1$ (see case (8) in the proof of Proposition \ref{prop:sporadicgroups}).

\subsection*{Acknowledgements} This research was supported by ARC Discovery Project DP150101066.

\section{Preliminaries} \label{sec:prelim}
For a group $G$ with subgroup $M$ we write $M \max G$ if $M$ is a maximal subgroup. Denote by $\frac{1}{2}G$ any subgroup of $G$ of index 2 if such a subgroup exists, and let $\frac{1}{2}G=G$ otherwise. For the subgroup structures in Table \ref{tab:Anspor} we follow the notation of \cite{Atlas}; in particular, for an integer $n$ and prime $r$ denote by $[n]$ an arbitrary group and $n$ a cyclic group of order $n$, and let $r^n$ denote an elementary abelian group of the same order. Moreover, we let $r_\epsilon^{1+2n}$ denote an extraspecial $r$-group of order $r^{1+2n}$ of type $\epsilon= \pm$. Furthermore, given groups $A$ and $B$ we use $A:B$ to denote a semidirect product of $A$ and $B$, where $A$ is normal. Finally, if a group $G$ acts on a set $\Omega$, then the permutation group induced by $G$ on $\Omega$ is denoted by $G^\Omega$.

Let $\Gamma=( V(\Gamma), E(\Gamma) )$ be a graph with $e=\{ u, v \} \in E(\Gamma)$ and let $G \leqslant \operatorname{Aut}(\Gamma)$.
 If $G$ acts primitively on the set of vertices or edges of $\Gamma$, we say $\Gamma$ is $G$-{\it vertex-primitive} or $G$-{\it edge-primitive} respectively. We call $\Gamma$ {\it vertex-primitive} if $\Gamma$ is $\Aut(\Gamma)$-vertex-primitive, and similarly if $\Gamma$ is $\Aut(\Gamma)$-edge-primitive. If $\Gamma$ is bipartite and $G$ acts transitively on the set of vertices of $\Gamma$ then $G$ has a normal subgroup $G^+$ of index 2 fixing each bipartite half setwise. We call a transitive group {\it biprimitive} if it is imprimitive and if all nontrivial systems of imprimitivity have precisely two parts. We say that $\Gamma$ is $G$-{\it vertex-biprimitive} if $G$ acts biprimitively on the set of vertices of $\Gamma$; furthermore, we say $\Gamma$ is {\it vertex-biprimitive} if $\Gamma$ is $\Aut(\Gamma)$-vertex-biprimitive. For an integer $s$, we say $\Gamma$ is $(G,s)$-{\it arc-transitive} if $G$ acts transitively on the set of $s$-arcs of $\Gamma$, and say $\Gamma$ is $s$-{\it arc-transitive} if $\Gamma$ is $(\Aut(\Gamma), s)$-arc-transitive.
 
  Let $\Gamma$ be a graph with $G\leqslant \Aut(\Gamma)$ acting transitively on the set of arcs of $\Gamma$. Then $G$ acts transitively on the set of 2-arcs of $\Gamma$ if and only if $G_v^{\Gamma(v)}$, the permutation group induced by $G_v$ on $\Gamma(v)$, is 2-transitive (see \cite[Lemma~9.4]{PLN}). This gives an easy test to see if an edge-primitive graph is 2-arc-transitive.  

 By \cite[Lemma~4.1]{eprim}, if $\Gamma$ is a disconnected edge-primitive graph then either $\Gamma$ is the union of isolated vertices and single edges or $\Gamma$ is the union of isolated vertices and a connected edge-primitive graph, and so to study edge-primitive graphs it suffices to consider connected edge-primitive graphs. By \cite[Lemma~3.4]{eprim}, a connected edge-primitive graph $\Gamma$ is either a star, a cycle, or $\Gamma$ is arc-transitive. 

 We say that a graph $\Gamma$ is a \emph{spread} of a graph $\Gamma_0$ if there is a partition $\mathcal{P}$ of $V(\Gamma)$ such that 
 \begin{enumerate}[(a)]
 \item $\Gamma_0$ is isomorphic to the quotient graph $\Gamma_{\mathcal{P}}$ with vertex set $\mathcal{P}$ and two parts $P_1,P_2\in\mathcal{P}$ are adjacent if there exist $v\in P_1$ and $w\in P_2$ such that $v$ and $w$ are adjacent in $\Gamma$; and
 \item if $P_1$ and $P_2$ are adjacent in $\Gamma_{\mathcal{P}}$ then there is a unique $v\in P_1$ and $w\in P_2$ such that $v$ and $w$ are adjacent in $\Gamma$.
 \end{enumerate}
 Note that there may be more than one $\Gamma_0$.  By \cite[Theorem~1]{eprim}, a nontrivial edge-primitive graph $\Gamma$ is either vertex-primitive, vertex-biprimitive or a spread of an edge-primitive graph $\Gamma_0$ that is vertex-primitive or vertex-biprimitive. If $\Gamma$ is a spread of $\Gamma_0$ and $\Aut(\Gamma)$ acts faithfully on the partition $\mathcal{P}$, then $\Aut(\Gamma) \leqslant \Aut(\Gamma_0)$ and the stabiliser in $\Aut(\Gamma)$ of a vertex of $\Gamma_0$ is the stabiliser of a part of the partition $\mathcal{P}$. For those graphs in Table \ref{tab:classicalgroupsfaithful} that are neither vertex-primitive nor vertex-biprimitive, the subgroup $S$ listed is such a stabiliser.

We now state the following characterisation of edge-primitive graphs as coset graphs.
\begin{lem}\cite[Proposition~2.5]{eprim} \label{prop:cosetgraph}
Let $G$ be a group with a maximal subgroup $E$. Then there exists a $G$-edge-primitive, $G$-arc-transitive graph $\Gamma$ with edge-stabiliser $E$ if and only if $E$ has a subgroup $A$ of index two, and $G$ has a core-free subgroup $H$ such that $A<H\neq E$; in this case $\Gamma\cong \Cos(G,H,HgH)$ for some $g\in E\backslash A$. 
\end{lem}

The following result reduces the study of edge-primitive graphs with $G, E$ and $H$ as in Lemma \ref{prop:cosetgraph} to $\Aut(G)$-conjugacy classes of $E$ and $H$. 
\begin{lem} \label{lem:isomorphicgraphs}
Let $G, E, A$ and $H$ be as in Lemma \ref{prop:cosetgraph}, and let $g, g_1 \in E \backslash A$ and $\phi \in \Aut(G)$. If $\Gamma = \Cos(G,H,HgH)$ and $\Gamma_1 = \Cos(G,H, Hg_1H)$ then $\Gamma = \Gamma_1$. Moreover, if $E_1 = E^\phi, A_1 = A^\phi, g_2 \in E_1 \backslash A_1, H_1 = H^\phi$ and $\Gamma_2=\Cos(G,H_1,H_1g_2H_1)$ then $\Gamma \cong \Gamma_2$. 
\end{lem}
\begin{proof}
Since $|E:A|=2$ we have that $g_1=ag$ for some $a\in A\leqslant H$. Hence $Hg_1H=HagH=HgH$ and so $\Gamma=\Gamma_1$. Moreover, the automorphism $\phi$ of $G$ gives a bijection from the set of right cosets of $H$ in $G$ to the set of right cosets of $H_1$ in $G$ and maps $HgH$ to $H_1g_2H_1$. Thus $\phi$ provides an isomorphism from $\Gamma$ to $\Gamma_2$.
\end{proof}

 For an almost simple group $G$ acting edge-primitively on a graph $\Gamma$, the next lemma allows us to consider the action of the subgroups $\soc(G) \leqslant G_1 \leqslant G$ on edges. 

\begin{lem} \label{lem:cosetgraphreduction}
Let $G$ be an almost simple group and let $\Gamma$ be a $G$-edge primitive, $G$-arc-transitive graph with edge-stabiliser $G_e$. If $\soc(G) \leqslant G_1 \leqslant G$ such that $\left( G_e \cap G_1 \right) \max G_1$, then $G_1$ is edge-primitive and arc-transitive.
\end{lem}
\begin{proof}
The special case $G_1 = \soc(G)=\PSL_2(q)$ is proved in \cite[Lemma~8.3]{eprim}, and the proof in the general case is nearly identical.
\end{proof}

 We now begin to consider the types of maximal subgroups $E$ of $G$ that yield edge-primitive graphs $\Gamma$ as in Lemma \ref{prop:cosetgraph}. The following results allow us to eliminate certain families of maximal subgroups. The first is essentially proved in \cite[Lemma~2.14]{GLX1}, and we omit the proof.

\begin{lem} \label{lem:Gvnormaliser}
Let $\Gamma$ be a connected $G$-arc-transitive graph, and let $\{ u, v \} \in E(\Gamma)$ with $g \in G$ such that $v^g=u$. Then any nontrivial normal subgroup of $G_v$ is not normalised by~$g$. 
\end{lem}

\begin{lem} \label{lem:cycliccentralizer}
Let $\Gamma$ be a nontrivial  $G$-edge-primitive graph and $e=\{u,v\} \in E(\Gamma)$. Assume $G_v$ acts unfaithfully on $\Gamma(v)$, and let $N \trianglelefteq G_e$ be a nonabelian normal subgroup such that $G_e/N$ is soluble and $N$ is a minimal normal subgroup of $G_{uv}$. Then $C_{G_e}(N)$ is noncyclic.
\end{lem}

\begin{proof}
Suppose for a contradiction that $C_{G_e}(N)$ is cyclic. Let $K \trianglelefteq G_{v}$ be the kernel of the action of $G_v$ on $\Gamma(v)$, so that $1\neq K \trianglelefteq G_{uv} \leqslant G_e$. As $N$ is a direct product of nonabelian simple groups, $G_e/N$ is soluble and $|G_e:G_{uv}|=2$, we have $N=G_e^{(\infty)}=G_{uv}^{(\infty)}$, where for a group $X$ we denote the last term in the derived series by $X^{(\infty)}$. As $K \cap N \leqslant N$ and $K \cap N \trianglelefteq G_{uv}$, either $N \leqslant K$ or $K \cap N = 1$. In the first case, this implies $G_e^{(\infty)} \leqslant K \leqslant G_e$, and so $K^{(\infty)}=G_e^{(\infty)} \trianglelefteq G_e$ and $K^{(\infty)} \operatorname{char} K \trianglelefteq G_v$. Thus $K^{(\infty})=N$ is a nontrivial normal subgroup of $G_v$ that is also normal in $G_e$, contradicting Lemma \ref{lem:Gvnormaliser}. In the second case this implies $K \leqslant C_{G_e}(N)$, and since $C_{G_e}(N)$ is cyclic we have $K \operatorname{char} C_{G_e}(N) \trianglelefteq G_e$. Thus $K$ is a nontrivial normal subgroup of $G_v$ that is also normal in $G_e$, again contradicting Lemma \ref{lem:Gvnormaliser}.
\end{proof}

 Using Lemma \ref{lem:cycliccentralizer}, we obtain the following useful result.
\begin{lem} \label{lem:GeGvalmostsimple}
Let $\Gamma$ be a nontrivial connected $G$-edge-primitive graph and $e=\{u,v\} \in E(\Gamma)$. If $G_v$ acts unfaithfully on $\Gamma(v)$, then neither $G_v$ nor $G_e$ is almost simple.
\end{lem}

\begin{proof}
 Let $K \trianglelefteq G_v$ be the kernel of the action of $G_v$ on $\Gamma(v)$ so that $K \trianglelefteq G_{uv}$, and suppose for a contradiction that $G_v$ is an almost simple group. Since $K \trianglelefteq G_v$ we have $K \geqslant \operatorname{soc}(G_v)$. As $K \trianglelefteq G_{uv} < G_v$, it follows that $G_{uv}$ is also an almost simple group with $\operatorname{soc}(G_{uv})=\operatorname{soc}(G_{v})$. Now $\operatorname{soc}(G_{uv})$ is a characteristic subgroup of $G_{uv}$, which in turn is normal in $G_e$. This  $\operatorname{soc}(G_v)=\operatorname{soc}(G_{uv})$ is normalised by both $G_v$ and $G_e$, contradicting Lemma \ref{lem:Gvnormaliser}.

 Next suppose that $G_e$ is an almost simple group. Then $G_{uv}$ is also an almost simple group with $\operatorname{soc}(G_{uv})=\operatorname{soc}(G_{e})$, and by the Schreier Conjecture, $N=\operatorname{soc}(G_{uv})$ satisfies the hypotheses of Lemma \ref{lem:cycliccentralizer}. However, $C_{G_e}(N)=1$ contradicting the fact that $C_{G_e}(N)$ must be noncyclic.
\end{proof}

 The next result is \cite[Lemma~4.1]{Lu2018}.
\begin{lem}\cite{Lu2018} \label{lem:lufaithful}
Let $\Gamma$ be a connected $d$-regular graph for $d \geqslant 3$, $e=\{u,v\} \in E(\Gamma)$ and $G \leqslant \operatorname{Aut}(\Gamma)$. If $\Gamma$ is $(G,2)$-arc-transitive and $G_v$ is faithful on $\Gamma(v)$, then $\Gamma$ is $(G,3)$-arc-transitive if and only if $d = 7, \operatorname{soc}(G_v) = \mathrm{A}_7$ and $G_e \ne \mathrm{S}_6$, i.e. $G_e = \PGL_2(9), \mathrm{M}_{10}$ or $\operatorname{Aut}(\mathrm{A}_6)$.
\end{lem}

\par Combining Lemmas \ref{lem:GeGvalmostsimple} and \ref{lem:lufaithful} provides a useful corollary:
\begin{cor} \label{cor:almostsimple}
Let $\Gamma$ be a nontrivial  $G$-edge-primitive graph with $e=\{ u,v \} \in E(\Gamma)$, and suppose that $\Gamma$ is $(G,3)$-arc-transitive. If $G_e$ or $G_v$ is an almost simple group, then $\operatorname{soc}(G_e)=\mathrm{A}_6, G_e \ne \mathrm{A}_6$ or $\mathrm{S}_6, \operatorname{soc}(G_v)=\mathrm{A}_7$ and $|G_v:G_{uv}|=7$.
\end{cor}

\begin{proof}
If the local action is unfaithful, then neither $G_e$ nor $G_v$ is almost simple by Lemma \ref{lem:GeGvalmostsimple}. If the local action is faithful, then the result follows from Lemma \ref{lem:lufaithful}.
\end{proof}

We present one final lemma which will prove useful in ruling out certain possibilities for~$G_v$.

\begin{lem} \label{lem:GvnotequalGe}
Let $\Gamma$ be a nontrivial  $G$-edge-primitive graph and $e = \{ u, v \} \in E(\Gamma)$. Then $|G_v|>|G_e|$.
\end{lem}

\begin{proof}
This is immediate from the fact that $\Gamma$ is nontrivial, as this implies the valency of $\Gamma$ is $|G_v:G_{uv}| \geqslant 3$.
\end{proof}

\par We now outline our method of proving Theorems \ref{thm:Anspor} and \ref{thm:faithful}. Let $\Gamma$ and $G$ be as in either theorem. By Lemmas \ref{prop:cosetgraph} and \ref{lem:isomorphicgraphs}, to classify such graphs it suffices to classify subgroup lattices $L=(G,E,A,H)$ up to $\Aut(G)$-conjugation, where $E \max G$, $A<E$ with $|E:A|=2$ and $A<H<G$ with $T=\soc(G) \not\leqslant H$. Moreover, by Corollary \ref{cor:almostsimple}, if either $E$ or $H$ is almost simple then both are with $\soc(E)=\mathrm{A}_6, \soc(H)=\mathrm{A}_7$, $|H:A|=7$ and $E \ne \mathrm{A}_6$ or $\mathrm{S}_6$. For each such lattice $L$, we consider the coset graph $\Gamma$ obtained from $L$. Clearly $G \leqslant \Aut(\Gamma)$, and so to determine if $G$ is the full automorphism group it suffices to check \cite{LPS} (which lists containments of primitive groups). If indeed $G = \Aut(\Gamma)$, we then check if $\Gamma$ is 3-arc-transitive: if either $E$ or $H$ is almost simple then both must be as in Lemma \ref{lem:lufaithful}, and if this is the case then $\Gamma$ is indeed 3-arc-transitive; if neither $E$ nor $H$ is almost simple, then we investigate the graph further. In most cases we can rule out the lattice $L$ by proving that the local action of $\Gamma$ cannot be 2-transitive (this is equivalent to showing that the action of $H$ on the set of right cosets of $A$ is not 2-transitive), and so $\Gamma$ is not 2-arc-transitive.

\section{Alternating groups} \label{sec:alternatinggroups}
 In this section, we consider the case where $G$ is almost simple with $\operatorname{soc}(G)=\mathrm{A}_n$ for $n \geqslant 5$.
The main result of this section is the following:

\begin{prop} \label{prop:alternatinggroups}
Let $\Gamma$ be a nontrivial connected edge-primitive 3-arc-transitive graph with $G=\operatorname{Aut}(\Gamma)$ and $\operatorname{soc}(G)=\mathrm{A}_n$ for some $n \geqslant 5$. Then $\Gamma$ is isomorphic to Tutte's 8-cage (described in Construction \ref{con:gamma0} below) and $G=\Aut(\mathrm{A}_6)$.
\end{prop}

\begin{construction}[Tutte's 8-cage] \label{con:gamma0} \sloppy Let $G=\operatorname{Aut}(\mathrm{A}_6)$. We define a coset graph  ${\Gamma_0=\operatorname{Cos}(G, H, HgH)}$ as in Lemma \ref{prop:cosetgraph} where $G, E, A$ and $H$ are as in Line 10 of Table \ref{tab:a6} and $g \in E \backslash A$. We note that $\Gamma_0$ is defined in \cite[Line 1, Table~2]{Li2011}, and from here we see that $\Gamma_0$ is 5-arc-transitive.
\end{construction}

\par To prove Proposition \ref{prop:alternatinggroups} we follow the method outlined at the end of Section \ref{sec:prelim}. Many of our preliminary lemmas are only for edge-primitive graphs instead of edge-primitive 3-arc-transitive graphs. We will use the following general set up.

\begin{hyp}\label{hyp:An}
Let $\Gamma=(V(\Gamma), E(\Gamma))$ be a nontrivial connected edge-primitive graph such that $G=\Aut(\Gamma)$ is an almost simple group with $\soc(G)=\mathrm{A}_n$. Let $e=\{u,v\}$ be an edge in $\Gamma$ and define the subgroups $E=G_e$, $H=G_v$ and $A=G_{uv}=E\cap H$. Note that $E\max G$, $|E:A|=2$ and $\soc(G) \not\leqslant H$.
\end{hyp}

 We proceed by considering each type of maximal subgroup $E$ of $G$. First assume $n \ne 6$. By the O'Nan Scott theorem, $E$ is either intransitive, imprimitive or primitive in its natural action, and if $E$ is primitive then it is either of affine, diagonal, product or almost simple type (see for example \cite{LPS}). The case where $E$ is intransitive or imprimitive is dealt with in Section \ref{sec:intransitive}, and the remaining cases are considered in Section \ref{sec:primitive}. The case where $n=6$ is different due to an exceptional outer automorphism, and we consider this case separately in Section \ref{sec:a6}.

\subsection{The case $\operatorname{soc}(G)=\mathrm{A}_6$}  \label{sec:a6}

\par Using \textsc{Magma} \cite{magma},  we list in Table \ref{tab:a6} all quadruples $(G, E, A, H)$ (up to $\Aut(G)$-conjugacy) such that $G$ is almost simple with $\operatorname{soc}(G)=\mathrm{A}_6$, $A<E \operatorname{max} G, |E:A|=2$ and $A \leqslant H$ where $\soc(G) \not\leqslant H$. In all cases we find that either $H \max G$ or $H \max \frac{1}{2}G$; we list the groups $S$ such that $H \max S$.

\begin{table}[!h]
\caption{Possible subgroup lattices for $\operatorname{soc}(G)=\mathrm{A}_6$} 
\label{tab:a6}
\centering
\begin{tabular}{c c c c c c} 
$G$ & $E$ & $A$ & $H$  &  $S$ & Notes \\ \hline \\[-2.5ex]
$\mathrm{A}_6$ & $\mathrm{S}_4$ & $\mathrm{A}_4$ & $\mathrm{A}_5$ & $G$  &\\
$\mathrm{S}_6$ & $\mathrm{S}_4 \times \mathrm{S}_2$ & $\mathrm{S}_4$ & $\mathrm{S}_5$ & $G$ &\\
$\mathrm{M}_{10}$ & $5:4$ & $5:2$ & $\mathrm{A}_5$ &  $\mathrm{A}_6$ \\
& $8:2$ & $\mathrm{D}_{8}$ & $\mathrm{S}_4$ & $\mathrm{A}_6$ \\
& $8:2$ & $\mathrm{Q}_8$ & $3^2:\mathrm{Q}_8$ & $G$ \\
$\PGL_2(9)$ & $\mathrm{D}_{20}$ & $\mathrm{D}_{10}$ & $\mathrm{A}_5$ & $\mathrm{A}_6$ \\
& $\mathrm{D}_{16}$ & $\mathrm{D}_8$ & $\mathrm{S}_4$ & $\mathrm{A}_6$ \\
& $\mathrm{D}_{16}$ & $8$ & $3^2:8$ & $G$ \\
$\operatorname{Aut}(\mathrm{A}_6)$ & $10:4$ & $\AGL_1(5)$ & $\mathrm{S}_5$  & $\mathrm{S}_6$ \\
& $[2^5]$ & $[2^4]$ & $\mathrm{S}_4 \times \mathrm{S}_2$ & $\mathrm{S}_6$ & Tutte's 8-cage \cite{Tutte47} \\
& $[2^5]$ & $[2^4]$ & $3^2:[2^4]$  & $G$ \\
 \hline
\end{tabular}
\end{table}

\begin{lem} \label{lem:a6}
Suppose that $\Gamma$ is as in Hypothesis \ref{hyp:An} with $n=6$. Then $\Gamma$ is isomorphic to either $\mathrm{K}_6$ or Tutte's 8-cage. In particular, if $\Gamma$ is also 3-arc-transitive, then $\Gamma$ is isomorphic to Tutte's 8-cage.
\end{lem}

\begin{proof}
\par We have that $(G,E, A, H)$ is isomorphic to a quadruple from Table \ref{tab:a6}, and from such a quadruple we can recover $\Gamma$ by Lemma \ref{prop:cosetgraph}. If $L$ is as in Line 1 or 2 then $G$ is 2-transitive on the set of right cosets of $H$, and so $\Gamma=\mathrm{K}_6$. Similarly, if $L$ is as in Line 5, 8 or 11 then $\Gamma \cong \mathrm{K}_{10}$. However, in this case $\Aut(\Gamma) = \mathrm{S}_{10}$. In lines 3, 6, and 9, $\Gamma$ is bipartite with bipartition stabiliser $\frac{1}{2}G$ isomorphic to $\mathrm{A}_6$ or $\mathrm{S}_6$. Moreover, if $\left\{ \Delta_1, \Delta_2 \right\}$ is the bipartition and $H$ is the stabiliser of a vertex in $\Delta_1$ then $H$ is transitive on $\Delta_2$. Thus $\Gamma \cong \mathrm{K}_{6,6}$. However, in this case $\operatorname{Aut}(\Gamma)=\mathrm{S}_6 \wr \mathrm{S}_2$.
\par For the remaining lines 4, 7 and 10 of Table \ref{tab:a6} we use \textsc{Magma} \cite{magma} to construct these graphs (using the CosetGeometry and Graph commands) and show that these graphs are isomorphic to Tutte's 8-cage as in Construction \ref{con:gamma0} above. This completes the proof.
\end{proof}

\subsection{Intransitive and imprimitive subgroups} \label{sec:intransitive}
By Section \ref{sec:a6}, we assume $n \ne 6$ for the remainder of this section.  We begin by considering the case where $E$ is an intransitive subgroup.

\begin{lem} \label{lem:intransitive}
Let $\Gamma$ be as in Hypothesis \ref{hyp:An} with $n\neq 6$ and suppose that  $E$ is intransitive, so that $E = \left( \mathrm{S}_k \times \mathrm{S}_{n-k} \right) \cap G$. Then $k=2$, $H \cong \mathrm{S}_{k-1} \cap G$ and $\Gamma \cong \mathrm{K}_n$.
\end{lem}

\begin{proof}
\par If $n=5$, then the result follows using \textsc{Magma} \cite{magma}. We may therefore assume $n \geqslant 7$.
\par We have $E= \left( \mathrm{S}_k \times \mathrm{S}_{n-k} \right) \cap G$ for some $1 \leqslant k <\frac{n}{2}$. Without loss, we may assume $E$ preserves the partition $\{ 1, \dots, k \} \cup \{ k+1, \dots, n \}$. Therefore $A \geqslant \mathrm{A}_k \times \mathrm{A}_{n-k}$ (where we define $\mathrm{A}_1=\mathrm{A}_2=1$). Hence $A$ contains a 3-cycle fixing at least 4 points, and so by a theorem of Jordan (see \cite[Theorem~3.3E]{DM}) $H$ is not primitive. Therefore $H$ is either intransitive or imprimitive.
\par First suppose that $H$ is imprimitive, so that $H \leqslant \mathrm{S}_a \wr \mathrm{S}_b$ for some $1< a, b \leqslant \frac{n}{2}$ such that $n=ab$. Observe that $A$ is $(n-k-2)$-transitive on $\{k+1, \dots, n\}$. In particular, $A$ is primitive on $\{k+1, \dots, n\}$, and so either $a>n-k$ or $b>n-k$, a contradiction.
\par We therefore have $H$ intransitive. If $k \ne 2$ then $A$ is transitive on $\{ 1, \dots, k\}$ and $\{ k+1, \dots, n \}$, and so $H \leqslant E$, contradicting Lemma \ref{lem:GvnotequalGe}. Therefore $k=2$ and $A$ is intransitive on $\{ 1, 2 \}$. As $A_{n-2}\leqslant A^{ \{ 3, \dots, n \} } $ and $|H| > |E|$ we must have $H = \left( \mathrm{S}_1 \times \mathrm{S}_{n-1} \right) \cap G$. Since $G$ is 2-transitive on the set of right cosets of $H$ this implies $\Gamma \cong \mathrm{K}_n$, completing the proof.
\end{proof}

\par We now consider the case where $E$ is a transitive and imprimitive subgroup. We first give an example of such a graph.

\begin{construction} \label{con:gamma1} Let $G=\mathrm{S}_8$ and $E \operatorname{max} G$ be an imprimitive subgroup of the form $E = \mathrm{S}_2 \wr \mathrm{S}_4$. Let $A=E \cap \mathrm{A}_8\cong C_2^3:S_4$. Then $|E:A|=2$ and there is a subgroup $H\cong\AGL_3(2)$ such that  $A < H < \mathrm{A}_8$. Thus $(G,E,A,H)$ is a quadruple as in Lemma \ref{prop:cosetgraph}. We therefore have a nontrivial connected $G$-edge-primitive arc-transitive graph $\Gamma_1=\operatorname{Cos}(G,H,HgH)$ for some $g \in E \backslash A$. We see that $\Gamma_1$ has order $|G:H|=30$ and valency $|H:A|=7$. Moreover, as $H<A_8$ and $|G:A_8|=2$ we have that $\Gamma_1$ is bipartite. Furthermore, the action of $H$ on the set of right cosets of $A$ is permutationally isomorphic to the action of $\GL(3,2)$ on 7 points, which is 2-transitive. Hence $\Gamma_1$ is 2-arc-transitive. However, using \textsc{Magma} \cite{magma} we find that $\Gamma_1$ has girth equal to 4 and so there is a 3-arc $(v_0,v_1,v_2,v_0)$. Since $\Gamma_1$ is bipartite and vertex-transitive, each bipartite half has size 15. As $\Gamma_1$ has valency 7 it follows that $\Gamma_1$ has diameter at least 3 and  so there is a 3-arc $(v_0,v_1,v_2,v_3)$ with $v_3\neq v_0$. Thus $\Gamma_1$ is not 3-arc-transitive.
\end{construction}

\begin{lem} \label{lem:imprimitive}
Let $\Gamma$ be as in Hypothesis \ref{hyp:An} with $n\neq 6$ and suppose that $E$ is imprimitive, so that $E=\left( \mathrm{S}_m \wr \mathrm{S}_k \right) \cap G$. Then $n=8$, $m=2$ and $\Gamma \cong \Gamma_1$ as in Construction \ref{con:gamma1}.
\end{lem}

\begin{proof}
\par Let $E= \left( \mathrm{S}_m \wr \mathrm{S}_k \right) \cap G$, where $1 < k,m \leqslant \frac{n}{2}$ and $n=mk$. Without loss of generality, we may assume that $E$ preserves blocks of imprimitivity of the form $\mathscr{B}_i=\left\{ (i-1)m+1, \dots, im \right\}$ for $1 \leqslant i \leqslant k$. Let $B = \left( \mathrm{S}_m^k \right) \cap G < E$ be the base of the wreath product. As neither $\mathrm{A}_m^k$ nor $\mathrm{A}_k$ contains a subgroup of index 2, $A \geqslant \mathrm{A}_m \wr \mathrm{A}_k$. 
\par First suppose $m>2$. As $A$ is transitive on $\{ 1, \dots, n\}$, so is $H$. Suppose $H$ is imprimitive, preserving a system of imprimitivity with blocks $\mathscr{C}_j$ for $1 \leqslant j \leqslant \ell$. Then $B \cap A$ preserves the partition $\mathscr{B}_i = \bigcup_{1 \leqslant j \leqslant \ell} \left( \mathscr{B}_i \cap \mathscr{C}_j \right)$ for all $1 \leqslant i \leqslant k$. But as $B \cap A \geqslant \mathrm{A}_m^k$ acts primitively on $\mathscr{B}_i$ we must have $\mathscr{B}_i \cap \mathscr{C}_j = \mathscr{B}_i$ or $\varnothing$ for all $1 \leqslant i \leqslant k$, or $|\mathscr{B}_i \cap \mathscr{C}_j|=1$ or $0$ for all $1 \leqslant i \leqslant k$. In the first case, since $\mathrm{A}_k$ acts primitively on the set of blocks $\left\{ \mathscr{B}_i \right\}$ we have $H=E$, contradicting Lemma \ref{lem:GvnotequalGe}. In the second case, if $ \mathscr{B}_{i_1} \cap \mathscr{C}_j \ne \varnothing$ and $\mathscr{B}_{i_2} \cap \mathscr{C}_j \ne \varnothing$ for some $i_1, i_2$ and $j$, then $B$ contains elements acting transitively on $\mathscr{B}_{i_1}$ whilst fixing $\mathscr{B}_{i_2} \cap \mathscr{C}_j$, implying $\mathscr{B}_{i_1} \subseteq \mathscr{C}_j$, a contradiction. Therefore for each $1 \leqslant j \leqslant \ell$, $\mathscr{B}_i \cap \mathscr{C}_j \ne \varnothing$ for exactly one value of $i$. But this implies that $|\mathscr{C}_j|=1$, a contradiction. Therefore $H$ is primitive. But $A$ contains a 3-cycle fixing at least 3 points, contradicting \cite[Theorem~3.3E]{DM}.
\par We must therefore have $m=2$, and so $k \geqslant 4$. Observe that $A \geqslant \mathrm{A}_k$, and so $A$ is transitive on the set of blocks $\mathscr{B}_i$ for $1 \leqslant i \leqslant k$; in particular, $\left( B \cap A \right)^{\mathscr{B}_i}\cong \left( B \cap A \right)^{\mathscr{B}_j}$ for all $i$ and $j$, and so we must have $\left( B \cap A \right)^{\mathscr{B}_i} \cong \mathrm{S}_2$ (otherwise $|E:A| \geqslant 2^{k-1}$). Hence $A$, and therefore $H$, is transitive on $\{ 1, \dots, n \}$. As before, $H$ is not imprimitive, and so $H$ is primitive. Observe that $E$ contains the elements $(1 \ 2)(3 \ 4), (1 \ 2)(5 \ 6)$ and $(3 \ 4)(5 \ 6)$, and so $A$ contains at least one of these elements since $|E:A|=2$. By \cite[Lemma~1.2]{Pr2}, we must therefore have $n=8$ and $H=\AGL_3(2)$. As $E \operatorname{max} G$, we must have $G=\mathrm{S}_8$ by \cite{LPS}. As $A = E \cap H$, we now have $(G,E,A,H)$ as in $\Gamma_1$ above, and so $\Gamma \cong \Gamma_1$. This completes the proof.  
\end{proof}

\subsection{Primitive subgroups} \label{sec:primitive}
 We now consider the case where $E$ is primitive. By the O'Nan-Scott theorem, $E$ is of affine, diagonal, product or almost simple type. We consider each type of primitive group, starting with the affine case.

\begin{lem} \label{lem:affine}
Let $\Gamma$ be as in Hypothesis \ref{hyp:An} with $n\neq 6$ and suppose further that $\Gamma$ is a 3-arc-transitive graph. Then $E$ is not of affine type.
\end{lem}

\begin{proof}
\par Suppose otherwise for a contradiction, and let $E=\AGL_d(p) \cap G$ where $n=p^d$. Let $N = \operatorname{soc}(E) \cong p^d$. As $N$ is the unique minimal normal subgroup of $E$, we have $N \leqslant A$. Therefore $A$ contains $N \rtimes \SL_d(p)$; in particular, $A$ is 2-transitive on $\{1,\ldots,n\}$ if $d >1$ and $A$ is primitive in all cases. By \cite{Pr1}, one of the following holds:
\begin{enumerate}
\item $H \leqslant \AGL_d(p)$ (a natural inclusion);
\item $H \leqslant \mathrm{S}_{n_0} \wr \mathrm{S}_m$ for some $n_0=p^e$ such that $n=n_0^m$ in product action (a blow-up of an exceptional inclusion);
\item $(A, H)$ is listed in \cite[Table~2]{Pr1} (an exceptional inclusion).
\end{enumerate}
\noindent In the first case we have $|H| \leqslant |E|$, contradicting Lemma \ref{lem:GvnotequalGe}, and the second case is not possible as either $A$ is 2-transitive or $n$ is a prime. Therefore $H$ lies in \cite[Table~2]{Pr1}. However, this implies that $H$ is an almost simple group not isomorphic to $\mathrm{A}_7$ or $\mathrm{S}_7$, contradicting Corollary \ref{cor:almostsimple}. 
\end{proof}

\par Next we consider the case where $E$ is of product type. We first require a preliminary result.

\begin{lem} \label{lem:productprelim}
Let $G$ be an almost simple group with $\soc(G)=\mathrm{A}_n$ and $n=m^2$ for some $m \geqslant 5$, and let $M \max G$ be of product type with $M = \left( \mathrm{S}_m \wr \mathrm{S}_2 \right) \cap G$. Suppose that  $Y=\frac{1}{2}M$ is imprimitive and $Y \leqslant X \leqslant G$ such that $X\neq M$. Then either  $X$ is primitive of almost simple type, or $X$ belongs to one of two chains
\begin{align*}
Y \max N  \max L \max G \ &\text{if $m$ is odd,} \\
Y \max N  \max S \max L \max G \  &\text{if $m$ is even,} 
\end{align*}
\noindent where $L=\left(\mathrm{S}_m \wr \mathrm{S}_m \right) \cap G$ is an imprimitive subgroup, 
$$
S=\left\{(g_1, \dots, g_m)h \in L\mid (g_1, \dots, g_m) \in \mathrm{S}_m^m, h \in \mathrm{S}_m, \sum_{1 \leqslant i \leqslant m} \mathrm{sgn}(g_i) \equiv 0 \pmod{2}\right\}
$$
and 
$$
N=\left\{(g_1, \dots, g_m)h \in L\mid (g_1, \dots, g_m) \in \mathrm{S}_m^m, h \in \mathrm{S}_m, \mathrm{sgn}(g_i) = \mathrm{sgn}(g_j) \textrm{ for all } 1 \leqslant i, j \leqslant m\right\}.
$$
\end{lem}

\begin{proof}
\par Suppose $Y \leqslant X \leqslant G$. Since $|M:Y|=2$, the group $Y$, and therefore $X$, acts transitively on $\{ 1, \dots, n \}$. As $Y$ is imprimitive, the projection of $Y$ onto $\mathrm{S}_2$ is trivial, and so $Y = \left( \mathrm{S}_m \times \mathrm{S}_m \right) \cap G$. Therefore $Y$ preserves two systems of imprimitivity $\{ \mathscr{B}^1_i \}$ and $\{ \mathscr{B}^2_i \}$ interchanged by $M \backslash Y$, each with $m$ blocks $\mathscr{B}^s_i$ of size $m$ for $1 \leqslant i \leqslant m$ and $1 \leqslant s \leqslant 2$, and so $Y$ is contained in two imprimitive maximal subgroups of the form $L = \left( \mathrm{S}_m \wr \mathrm{S}_m \right) \cap G$. Note that one factor of $\mathrm{S}_m \times \mathrm{S}_m$ occurs as a diagonal subgroup of the base group $\mathrm{S}_m^m$, and the other permutes the factors. Observe that the suborbits of $Y$ have lengths 1, $m-1$ (occurs twice) and $(m-1)^2$, and that the two suborbits of length $m-1$ correspond to the two systems of imprimitivity.
\par As $Y \cap \mathrm{S}_m^m$ is a diagonal subgroup, $Y$ is contained in $N$, and since $m \geqslant 5$ and $Y$ acts primitively on the $m$ direct factors of $\mathrm{S}_m^m$ it is elementary to show that $Y$ is in fact maximal in $N$. Also, $T=\mathrm{A}_m \wr \mathrm{S}_m \leqslant N \leqslant L$, and subgroups of $L$ containing $T$ correspond to subgroups of $\mathrm{S}_m^m / \mathrm{A}_m^m \cong \mathrm{C}_2^m$ normalised by $\mathrm{S}_m$, and these correspond to submodules of the permutation module $\mathbb{F}_2^m$ of $\mathrm{S}_m$. The only nonzero proper submodules of $\mathbb{F}_2^m$ are the constant module (corresponding to $N$) and the deleted permutation module (corresponding to $S$), and $N \leqslant S$ if and only if $m$ is even. Therefore $N \max L$ if $m$ is odd and $N \max S \max L$ if $m$ is even. Moreover, $Y < S$ if and only if $m$ is even. Finally $L \max G$ by \cite{LPS}, and so $Y$ is indeed contained in the two maximal chains stated. It remains to prove that, other than $M$ and almost simple primitive groups, there are no other overgroups of $Y$.
\par First consider the case where $X$ acts imprimitively on $\{ 1, \dots, n \}$, and assume that $X$ preserves a system of imprimitivity with blocks $\mathscr{C}_j$ for $1 \leqslant j \leqslant \ell$. Fix one of the two systems of imprimitivity $\{ \mathscr{B}_i \}=\{ \mathscr{B}^s_i \}$ described above. Arguing in an identical manner to the proof of Lemma \ref{lem:imprimitive}, we find that one of the following holds:
\begin{enumerate}
\item $|\mathscr{B}_i \cap \mathscr{C}_j| = 0$ or $m$ for all $1 \leqslant i \leqslant m$ and all $1 \leqslant j \leqslant \ell$;
\item $|\mathscr{B}_i \cap \mathscr{C}_j| = 0$ or 1 for all $1 \leqslant i \leqslant m$ and all $1 \leqslant j \leqslant \ell$.
\end{enumerate}
\noindent In case (1), since $Y$ acts primitively on the set of blocks $\{ \mathscr{B}_i \}$, for each $1 \leqslant i \leqslant m$ we have $\mathscr{B}_i=\mathscr{C}_j$ for some $j$, and so $X\leqslant L$. Now suppose (2) holds. Note that this implies $m \leqslant \ell$, since $\mathscr{B}_i = \bigcup_{1 \leqslant j \leqslant \ell} \left(\mathscr{B}_i \cap \mathscr{C}_j \right)$. Let $\alpha_1 \in \mathscr{B}_{i_1} \cap \mathscr{C}_j$ and $\alpha_2 \in \mathscr{B}_{i_2} \cap \mathscr{C}_j$. Recall that the suborbits of $Y$ have length 1, $m-1$ (occurs twice) and $(m-1)^2$. If $\alpha_2$ belongs to the suborbit of length $(m-1)^2$ then this implies $|\mathscr{C}_j| \geqslant (m-1)^2$, a contradiction since $(m-1)^2 > \frac{1}{2}m^2$ for $m \geqslant 5$. Therefore $\alpha_2$ belongs to one of the two suborbits of length $m-1$, and so $|\mathscr{C}_j| \geqslant m$. Since $m \leqslant \ell$ we have $m = \ell = |\mathscr{C}_j|$, and this implies that the system of imprimitivity $\{ \mathscr{C}_j\}$ is either equal to $\{ \mathscr{B}_i \}$ or the image of $\{\mathscr{B}_i\}$
 by an element of $M \backslash Y$. As we are assuming (2) the second case holds, and this completes the proof that the only systems of imprimitivity preserved by $Y$ are $\{ \mathscr{B}^1_i \}$ and $\{ \mathscr{B}^2_i \}$.
\par We now consider the case where $X$ acts primitively on $\{ 1, \dots, n \}$. Assume for the moment that $X \max G$, so that $X$ is of affine, diagonal, product or almost simple type. First suppose that $X$ is affine, so that $n=p^d$ for some prime $p$ and $X=\AGL_d(p) \cap G$. Let $\alpha \in \{ 1, \dots, n \}$ so that $\mathrm{S}_{m-1}^2 \cap Y \cong Y_\alpha \leqslant \GL_d(p)$. Let $\beta \in \{ 1, \dots, n \}$ be contained in an orbit of $Y_\alpha$ of size $m-1$. Then $Y_{\alpha \beta}$ has orbits of size 1 (occurs twice), $m-2, m-1$ (occurs twice) and $(m-1)(m-2)$. However, if $\beta$ corresponds to a non-zero vector $v_\beta \in V_d(p)$, the natural module of $\GL_d(p)$, then $\GL_d(p)_\beta$ has $p$ orbits of size 1, namely $\{ \lambda v_\beta \}$ for $\lambda \in \mathbb{F}_p$. This implies $p=2$. Now choose $\gamma$ from the same suborbit as $\beta$. Then $Y_{\alpha \beta \gamma}$ has orbits of length 1 (occurs thrice), $m-3, m-1$ (occurs thrice) and $(m-1)(m-3)$. However, if $\gamma$ corresponds to $v_\gamma \in V_d(p)$, $\GL_d(p)_{\beta \gamma}$ fixes the 4 points $0=v_\alpha, v_\beta, v_\gamma$ and $v_\beta+v_\gamma$, and this yields a contradiction. Hence $X$ is not of affine type.
\par Suppose now that $X$ is of diagonal type, so that $X=\left( T^k.\left( \Out(T) \times \mathrm{S}_k \right) \right) \cap G$ for some nonabelian simple group $T$ and $k \geqslant 2$ with $n=|T|^{k-1}$. Note that in this case we have $m > 9$, as otherwise this forces $n=|\mathrm{A}_5|=60$ which is not square. Observe that $Y \geqslant \mathrm{A}_m^2$, and so $Y$ contains a 3-cycle in one of the factors which permutes exactly $3m$ points. By \cite[Theorem~2]{LS}, the minimal degree of $X$ (that is, the smallest number of points moved by any nontrivial element of $X$) is $\mu(X) \geqslant \frac{1}{3}n$. However, this is a contradiction since $3m<\frac{1}{3}n$. Hence $X$ is not of diagonal type.
\par We now suppose that $X$ is of product type, with $X=\left( \mathrm{S}_\ell \wr \mathrm{S}_k \right) \cap G$ for $\ell \geqslant 5, k \geqslant 2$ and $n=\ell^k$. Now $X$ has suborbits of lengths 1 and $(\ell-1)^k$, and so the remaining suborbits have lengths at most $\ell^k-(\ell-1)^k -1$. Recall that $Y$ has suborbits of lengths $1, m-1$ (occurs twice) and $(m-1)^2$. This forces $k=2$, as otherwise $(m-1)^2 > 1, (\ell-1)^k$ and $\ell^k - (\ell-1)^k -1$. Hence $X$ is of the same type as $M$, and we have $X = N_{G}(Y)=M$. 
\par We have therefore shown that if $Y$ is contained in any primitive group $X$, then either $X$ is of almost simple type or $X=M$. This completes the proof.
\end{proof}

\begin{lem} \label{lem:product}
Let $\Gamma$ be as in Hypothesis \ref{hyp:An} and suppose further that $\Gamma$ is a 3-arc-transitive graph. Then $E$ is not of product type.
\end{lem}

\begin{proof}
\par Let $E$ be of product type, so that $E= \left( \mathrm{S}_m \wr \mathrm{S}_k \right) \cap G$ where $n=m^k$ for $m \geqslant 5$ and $k \geqslant 2$. Then $A \geqslant \mathrm{A}_m \wr \mathrm{A}_k$, and so either $A$ is primitive of product type or $k=2$. First suppose that $A$ is primitive. With the terminology of \cite{Pr1}, one of the following holds:
\begin{enumerate}
\item $H \leqslant \mathrm{S}_m \wr \mathrm{S}_k$  (a natural inclusion);
\item $m=8$ and $A \leqslant \PSL_2(7) \wr \mathrm{S}_k$ (an exceptional inclusion as listed in \cite[Table~1]{Pr1});
\item $m=8^\ell$ for some $\ell > 1$ and $A \leqslant (\PSL_2(7) \wr \mathrm{S}_\ell) \wr \mathrm{S}_k$ (a blow-up of an exceptional inclusion). 
\end{enumerate}
\noindent As $H$ is not contained in $E$ by Lemma \ref{lem:GvnotequalGe}, the first case does not hold. However, it is clear that the second and third cases also cannot hold as $\mathrm{A}_m \wr \mathrm{A}_k \leqslant A$, a contradiction. 
\par Therefore $A$ is not primitive, and so we must have that $k=2$ and $A$ projects trivially onto $\mathrm{S}_2$. Thus $A$ is as in Lemma \ref{lem:productprelim}. Since $H \not\leqslant E$ and $H$ is not almost simple by Lemma \ref{lem:GeGvalmostsimple} (as $|A|>|\mathrm{S}_6|$), $H$ is imprimitive and belongs to one of the two chains described in the lemma. However, the action of $H$ on the set of cosets of $A$ is not 2-transitive, a contradiction. Hence $E$ is not of product type. \end{proof}

\begin{lem} \label{lem:diagonal}
Let $\Gamma$ be as in Hypothesis \ref{hyp:An}. Then $E$ is not of diagonal type.
\end{lem}

\begin{proof}
\par Let $E$ be of diagonal type with $E= \left( T^k.(\operatorname{Out}(T) \times \mathrm{S}_k) \right) \cap G$ for a nonabelian simple group $T$ such that $k \geqslant 2$ and $n=|T|^{k-1}$. Then $A \geqslant T^k$, and the image of $A$ under the projection $E \rightarrow \mathrm{S}_k$ contains $\mathrm{A}_k$. Hence $A$ is also a primitive group of diagonal type. By \cite{Pr1}, $H$ is contained in $N_G(T^k)=E$, contradicting Lemma \ref{lem:GvnotequalGe}. Hence $E$ is not of diagonal type. \end{proof}

\par We now consider the case where $E$ is almost simple. We require the following lemma:

\begin{lem} \label{lem:a6a7maximal}
Let $G$ be an almost simple group with $\operatorname{soc}(G)=\mathrm{A}_n$ for $n \geqslant 5$, and suppose $G$ has subgroups $M$ and $N$ satisfying the following:
\begin{enumerate}
\item $M$ and $N$ are both almost simple, with $\operatorname{soc}(M)=\mathrm{A}_7$ and $\operatorname{soc}(N)=\mathrm{A}_6$;
\item $\operatorname{soc}(G) \not\leqslant M$;
\item $N \operatorname{max} G$;
\item $\soc(N) \leqslant M$.
\end{enumerate}
\noindent Then $G=\mathrm{A}_8, M \cong \mathrm{A}_7$ and $N \cong \mathrm{S}_6$.
\end{lem}

\begin{proof}
\par First suppose $n \geqslant 10$. We consider $G$ as a permutation group on $n$ letters in the natural way. If $N$ is an intransitive subgroup, then for some $m$ and $k$ such that $n=m+k$ we have $\left( \mathrm{S}_m \times \mathrm{S}_k \right) \cap \mathrm{A}_n \leqslant N$, a contradiction since $\frac{1}{2}|\mathrm{S}_m \times \mathrm{S}_k| > |\Aut(\mathrm{A}_6)|$ for $n \geqslant 10$. If $N$ acts imprimitively, then for some $m, k$ such that $n=mk$ we have $\left( \mathrm{S}_k^m \right) \cap \mathrm{A}_n \trianglelefteq N$, a contradiction. Therefore $N$ acts primitively, and so $\soc(N)$ is transitive. Suppose for the moment that $M$, and therefore $\soc(N)$, acts imprimitively. Using \textsc{Magma} \cite{magma}, we find that this is only possible if $N \ne \mathrm{S}_6$ and $n=36$ or 45; however, $M$ has no transitive action on $n$ points, a contradiction. Therefore $M$ also acts primitively. The only possible degree of a primitive action for both $M$ and $N$ is 15 (this can easily be seen using \textsc{Magma}). However, neither $\mathrm{S}_{15}$ nor $\mathrm{A}_{15}$ contains $\mathrm{S}_7$ or $\mathrm{A}_7$ as a maximal subgroup. 
\par Therefore $n<10$. Clearly $n \geqslant 8$, and the result follows using \textsc{Magma}. 
\end{proof}

\begin{cor} \label{cor:almostsimplealternatinggroups}
Let $\Gamma$ be as in Hypothesis \ref{hyp:An} and suppose further that $\Gamma$ is a 3-arc-transitive graph. Then $E$ is not an almost simple primitive group.
\end{cor}

\begin{proof}
Suppose $E$ is an almost simple group. Then by Corollary \ref{cor:almostsimple}, $\operatorname{soc}(E)=\mathrm{A}_6$ and $H$ is also an almost simple group with $\operatorname{soc}(H)=\mathrm{A}_7$. Therefore $G, E$ and $H$ satisfy the hypothesis of Lemma \ref{lem:a6a7maximal}. However, this implies $E=\mathrm{S}_6$, contradicting Corollary \ref{cor:almostsimple}.
\end{proof}

\subsection{Proof of Proposition \ref{prop:alternatinggroups}}
\begin{proof}[Proof of Proposition \ref{prop:alternatinggroups}]
\par If $n=6$ then the result follows from Lemma \ref{lem:a6}. We may therefore assume $n \ne 6$. We follow the method outlined at the end of Section \ref{sec:prelim} and classify lattices $L=(G,E,A,H)$. By the O'Nan Scott theorem, $E \operatorname{max} G$ is either intransitive, imprimitive or primitive, and if $E$ is primitive then $E$ is of affine, diagonal, product or almost simple type. Each of these types is considered in Lemmas \ref{lem:intransitive} to \ref{lem:diagonal} and Corollary \ref{cor:almostsimplealternatinggroups}, and the result follows as $\mathrm{K}_n$ is not 3-arc-transitive.
\end{proof}

\section{Sporadic groups} \label{sec:sporadicgroups}
\par In this section, we consider the case where $G$ is an almost simple sporadic group and complete the proof of Theorem \ref{thm:Anspor}. As usual we let $\Gamma=(V(\Gamma), E(\Gamma))$ be a nontrivial connected edge-primitive graph with $G=\operatorname{Aut}(\Gamma), e=\{u,v\} \in E(\Gamma)$, vertex-stabiliser $G_v$ and edge-stabiliser $G_e$ so that $G_e \operatorname{max} G$ and $G_{uv} = G_v \cap G_e$ is a subgroup of $G_e$ of index~2. 
\par The main result of this section is the following:

\begin{prop} \label{prop:sporadicgroups}
Let $\Gamma$ be a nontrivial connected edge-primitive 3-arc-transitive graph with $G=\operatorname{Aut}(\Gamma)$ an almost simple sporadic group. Then $\Gamma=\Cos(G,H,HgH)$ where $g \in E \backslash A$ and $(G,E,A,H)$ are listed in Table \ref{tab:sporadicgroups} below.
\end{prop}

\begin{rem}
In Table \ref{tab:sporadicgroups} we record the group $S \leqslant G$ where $H \max S$ and either $S= G$ or $|G:S|=2$. We note that the first three graphs in Table \ref{tab:sporadicgroups} are in fact 4-arc-transitive and listed in \cite[Table~2]{Li2011}. The fourth graph is \cite[Example~4.2]{Lu2018}.
\end{rem}

\begin{table}[!h]
\caption{Edge-primitive 3-arc-transitive graphs with $\Gamma$ and $G$ as in Proposition \ref{prop:sporadicgroups}}
\label{tab:sporadicgroups}
\centering
\begin{tabular}{c c c c c c c} 
$G$ & $E$ & $A$ & $H$ & $S$ & Notes \\ \hline \\[-2.5ex]
$\mathrm{M}_{12}.2$ &  $3^{1+2}_+:\mathrm{D}_8$ & $3^{1+2}_+:2^2$ & $3^2:2\mathrm{S}_4$ & $\mathrm{M}_{12}$ & Weiss \cite{Weiss85} \\
$\mathrm{J}_3.2$ &  $[2^6]:(\mathrm{S}_3)^2$ & $[2^6]:\left( (\mathrm{S}_3)^2 \cap \mathrm{A}_6 \right)$& $[2^4]:(3 \times \mathrm{A}_5).2$ & $G$ & Weiss \cite{Weiss86}\\
$\mathrm{Ru}$ & $5^{1+2}_+:[2^5]$ & $5^{1+2}_+:[2^4]$ &  $5^2:\GL_2(5)$ & $G$ & Ru graph \cite{StrWe} \\
$\mathrm{O'N}.2$ & $\PGL_2(9)$ & $\mathrm{A}_6$ & $\mathrm{A}_7$  & $\mathrm{O'N}$ & Lu \cite{Lu2018}  \\
 \hline
\end{tabular}
\end{table}

\begin{proof}[Proof of Proposition \ref{prop:sporadicgroups}]
\par We follow the method outlined at the end of Section \ref{sec:prelim} and classify lattices $L=(G,E,A,H)$ up to $\Aut(G)$-conjugation. Let $G$ be an almost simple sporadic group with socle $T$, and suppose $G$ contains subgroups $E, A$ and $H$ as in Lemma \ref{prop:cosetgraph}. Information on the maximal subgroups of $G$ can almost always be found in the Atlas \cite{Atlas}, and a more complete list can be found in the survey article \cite{Wi1}. All maximal subgroups of $G$ are known except if $G=\mathrm{M}$ is the Monster group; however, by \cite{Wi1} the remaining possibilities for maximal subgroups $E$ of $G$ are all almost simple with $\soc(E) = \PSL_2(13)$ or $\PSL_2(16)$, and so by Corollary \ref{cor:almostsimple} these cannot be edge-stabilisers of an edge-primitive 3-arc-transitive graph. Moreover, the order of a subgroup of index 2 of a maximal subgroup $E \max \mathrm{M}$ does not divide the order of an almost simple group with one of these socles, and so the vertex-stabiliser $H$ of an edge-primitive graph with automorphism group $\mathrm{M}$ and edge-stabiliser $E$ cannot be contained in any of these unknown maximal subgroups.
\par We now outline our method for finding lattices $L$.
\begin{enumerate}
\item Using \cite{Wi1}, for each $\Aut(T)$-class representative $E$ of maximal subgroups of $G$ we list the pairs of subgroups $(E,M)$ of $G$ such that $M \max G$ or $M \max \frac{1}{2}G$ and $|M|$ is divisible by $|E|/2$. (It is perhaps interesting to note that all almost simple sporadic groups contain at least one such pair). We suppose that $H \leqslant M$.
\item If $E$ does not contain a subgroup of index 2 then $E$ does not occur in a lattice as above, and we remove the pairs containing $E$.
\item If $E$ is almost simple, then we must have $E$ and $H$ as in Corollary \ref{cor:almostsimple} with $H \leqslant M$ for some $M$. If this is the case, and $E$ contains a subgroup of index 2 contained in $H$, then $\Gamma$ is indeed 3-arc-transitive by Lemma \ref{lem:lufaithful} and we list $L=(G,E,A,H)$ in Table \ref{tab:sporadicgroups} (this yields the fourth row of Table \ref{tab:sporadicgroups}). Otherwise this does not yield a lattice, and we remove $(E,M)$. 
\item For each $(E, M)$ we consider the simple sections of $M$. Suppose $Q \trianglelefteq R \leqslant M$ with $X = R/Q$ a simple group. For each subgroup $A$ of $E$ of index 2, if $A \leqslant M$ then we have $\left( A \cap R \right)/\left( A \cap Q \right) \cong \left(A \cap R \right)Q/Q \leqslant X$. Using \textsc{Magma}, \cite{Atlas} and \cite{Wi1} (and possibly \cite{NW} and \cite{Wi2}), in most cases we find that there exists a simple section $X$ of $M$ such that $X$ contains no subgroup of the form $\left( A \cap R \right)/\left( A \cap Q \right)$ for all subgroups $A$ of $E$ of index 2. If this is the case, we remove the pair $(E,M)$ from our list.
\item For each remaining pair $(E, M)$ we list the pairs of subgroups $(E,H)$ where $H \leqslant M$ is a subgroup that possibly contains $A$. We repeat steps (3) and (4) for $H$: if $H$ is almost simple then we must have $E$ and $H$ as in Corollary \ref{cor:almostsimple}, and we also consider the simple sections of $H$. If $(E,H)$ fails either of these steps then we remove $(E,H)$. 
\item \sloppy We now suppose that each remaining pair $(E,H)$ gives rise to a lattice $L=(G, E, A, H)$ as above. We consider the action of $H$ on the set of right cosets of $A$. If $\Gamma$ is 2-arc-transitive then this action is 2-transitive. Therefore, using \cite[Tables~7.3, 7.4]{Ca}, if it is not possible for $H$ to have a 2-transitive action on the set of right cosets of $A$ then we remove the pair $(E,H)$. It is useful to note that if $d=|H:A|$ and $H$ acts 2-transitively on the set of right cosets of $A$, then $d(d-1)$ divides $|H|$. This fact can often be used to rule out certain pairs $(E,H)$. 
\item In a number of remaining cases $(E,H)$ the valency equals 4, and we check the classification of tetravalent edge-primitive 2-arc-transitive graphs given in \cite{Guo2014b} to see if such a lattice exists.
\item This leaves the lattices listed in \cite[Table~2]{Li2011} (which we add to our table) and one other possible lattice $L=(\mathrm{J}_1,7:6,7:3,2^3:7:3)$. Using \textsc{Magma}, we find that $L$ does indeed exist and gives rise to an edge-primitive 2-arc-transitive graph $\Gamma$. However, in this case $\Gamma$ is not 3-arc-transitive, as the valency is $|H:A|=8$ and the stabiliser of a 2-arc has order 3. This completes Table \ref{tab:sporadicgroups}.
\end{enumerate}
\par We note that if $G$ is sufficiently small then all lattices $L$ of the above form can be completely determined using \textsc{Magma}, and in general this is much faster than using the method outlined above.

\par As an example, we show that $G=\mathrm{Co_1}$ is not the automorphism group of an edge-primitive 3-arc-transitive graph using the method above (note that generators of all maximal subgroups of $G$ are not known, and so in this case we cannot simply use \textsc{Magma} to determine all lattices). By \cite[p. 68]{Wi1}, there are 22 classes of maximal subgroups of $G$. For ease of notation, we refer to a class of maximal subgroups, or a representative of a class, by its number using the ordering given in \cite[p. 68]{Wi1}.
\begin{enumerate}
\item  The orders of the maximal subgroups are listed in \cite{Atlas} (note the corrections given in \cite{Wi1}). We find that there are 33 pairs $(E,M)$ where $E \max G, M \max G$ and $|E|/2$ divides $|M|$. 
\item Without knowing further information on the structure of the maximal subgroups, we cannot rule out any of our pairs by (2).
\item We cannot rule out any pairs by (3), as $E$ is not equal to group number 1, 4 or 6.
\item For each $(E,M)$ we consider the simple sections of $M$. For example, consider the pair (7,2). Here $E=(\mathrm{A_4} \times \mathrm{G}_2(4)):2$ and $M=3.\mathrm{Suz}:2$, and so $A=\mathrm{A_4} \times \mathrm{G}_2(4)$. Let $Q=3 \trianglelefteq R=3.\mathrm{Suz} \leqslant M$ and $X=R/Q \cong \mathrm{Suz}$. Then $A \cap R=A$, and since $A$ has no normal subgroup of order 3 we have $A \cap Q=1$ and so $|\left( A \cap R \right) : \left( A \cap Q \right)|=2^{14}.3^4.5^2.7.13$. However, this yields a contradiction using \cite{Atlas}, as no maximal subgroup of $X$ has order divisible by $2^{14}.3^4.5^2.7.13$. We therefore remove the pair (7,2). Using similar reasoning, we remove all but one pair.
\item We are left to consider the pair (18,12), where $E \cong \left( \mathrm{D_{10}} \times \left( \mathrm{A_5} \times \mathrm{A_5} \right).2 \right).2$ and $M \cong \left( \mathrm{A_5} \times \mathrm{J}_2 \right):2$. Suppose $H \leqslant M$. Let $Q=\mathrm{A}_5$ and $R=\mathrm{A}_5 \times \mathrm{J}_2 \leqslant M$ so that $Q \trianglelefteq R$ and $X=R/Q=\mathrm{J}_2$. Then $Y=\left( A \cap R \right) / \left(A \cap Q \right)$ satisfies $2^3.3.5^2 \leqslant |Y| \leqslant |A|$. From \cite{Atlas} we must have $|Y|=2^3.3.5^2$ and $Y \max X$, and so either $\left( H \cap R \right) / \left(H \cap Q \right) = Y$ or $\left( H \cap R \right) / \left(H \cap Q \right) = X$. The first case implies $H=A$, a contradiction. The second case implies that $H=M$, and so this is the only possibility for the case (18,12). 

\item \sloppy We now suppose $(E,H)$ gives rise to a lattice $(G, E, A, H)$ for some $E \cong \left( \mathrm{D_{10}} \times \left( \mathrm{A_5} \times \mathrm{A_5} \right).2 \right).2$ and $H \cong \left( \mathrm{A_5} \times \mathrm{J}_2 \right):2$ and some subgroup $A < E$ of index 2. Then $d=|H:A|=2^4.3^2.7$. Observe that $|H|$ is not divisible by $d(d-1)$, and so the action of $H$ on the set of right cosets of $A$ is not 2-transitive. Therefore this lattice (if it exists) does not yield a 2-arc-transitive graph. This removes the final pair, and completes the proof for $G=\mathrm{Co}_1$.
\end{enumerate} 
The remaining cases are proved in a similar manner.
\end{proof}

\begin{proof}[Proof of Theorem \ref{thm:Anspor}]
This follows immediately from Propositions \ref{prop:alternatinggroups} and \ref{prop:sporadicgroups}.
\end{proof}

\section{Classical groups: faithful local action} \label{sec:classicalgroupsfaithful}

\par In this section we prove Theorem \ref{thm:faithful}. Throughout we assume that $G = \Aut(\Gamma)$ is a finite almost simple classical group such that $G_v$ acts faithfully on $\Gamma(v)$. We begin by fixing some notation. Let $T=\soc(G)$ be a finite simple classical group with natural module $V$ of dimension $n$ over $\mathbb{F}_{q^\delta}$ of characteristic $p$, where $\delta=2$ if $T$ is unitary and $\delta=1$ otherwise. For the notation of $T$ and $G$ we follow \cite[\S~2]{KlLi}: in particular, we let $\GammaL_n(q^\delta)$ be the semilinear group and $\PGammaL_n(q^\delta)$ its projective version. To refer to outer automorphisms of $T$ we use the conventions of \cite[\S~1.7.1]{BHRD}: we denote by $\delta$ a diagonal automorphism, $\phi$ a field automorphism and $\gamma$ a graph automorphism (in the sense of the Dynkin diagram). 
\par Aschbacher's theorem \cite{As} states that if $X \leqslant G$ and $G$ does not contain an exceptional outer automorphism (in the case where $T=\PSp_4(q)$ and $q$ is even or $T=\POmega_8^+(q)$) then either $X$ is contained in a member of an Aschbacher class $\mathscr{C}_i$ for some $1 \leqslant i \leqslant 8$ or $X \in \mathscr{C}_9$. Details of the structure of the Aschbacher classes can be found in \cite[\S~4]{KlLi}. In particular, we note that if $X \in \mathscr{C}_9$ then $X$ is an almost simple group such that $\soc(X)$ acts absolutely irreducibly on the natural module $V$ of $G$, and $\soc(X)$ is not contained in a member of $\mathscr{C}_i$ for $i=1,3,5$ or 8. This will be useful in the following lemma.

\begin{lem} \label{lem:a6a7maximalclassical}
Let $G$ be a finite almost simple classical group with subgroups $M$ and $N$ satisfying the following:
\begin{enumerate}
\item $M$ and $N$ are almost simple, with $\soc(M)=\mathrm{A_7}$ and $\soc(N)=\mathrm{A_6}$;
\item $N \max G$;
\item $\soc(G) \not\leqslant M$;
\item $\soc(N) \leqslant M$.
\end{enumerate}
\noindent Then $\soc(G)$ is isomorphic to $\PSU_3(5), \PSL_4(2), \PSL_6^\epsilon(q)$ or $\POmega_{10}^-(7)$.
\end{lem}

\begin{proof}
\par Let $T=\soc(G)$ be as above. Also let $X=\soc(M)=\mathrm{A_7}$ and $Y=\soc(N)=\mathrm{A_6}$. For a subgroup $S \leqslant G \leqslant \PGammaL_n(q^\delta)$, denote by $\hat{S} \leqslant \GammaL_n(q^\delta)$ the full preimage of $S$. Conversely, for a subgroup $S \leqslant \GammaL_n(q^\delta)$ let $\overline{S}=S/\left(S \cap Z(\GammaL_n(q^\delta)) \right) \leqslant \PGammaL_n(q^\delta)$. Observe that $N/(N \cap T) \leqslant G/T$ which is soluble by the Schreier conjecture, and so $T \geqslant Y$. Similarly $T \geqslant X$. 
\par First suppose $n \leqslant 9$. The maximal subgroups of $G$ are listed in \cite[Tables~8.1--8.59]{BHRD}. In the case where Aschbacher's theorem applies, we first prove that
\begin{align} \label{eqn:YinS}
\text{if $N \in \mathscr{C}_9$ then $M \in \mathscr{C}_9$.}
\end{align}
\noindent Suppose for a contradiction that $N\in\mathscr{C}_9$ and $M \notin \mathscr{C}_9$. From the definition of $\mathscr{C}_9$ and the fact $\soc(N) \leqslant M$, we see that $M$ is not contained in a member of $\mathscr{C}_i$ for $i=1,3,5$ or 8. If $M$ is contained in a member of $\mathscr{C}_2$ then $M$ preserves a subspace decomposition $\mathscr{D}$ of the form $V=U_1 \oplus \dots \oplus U_t$ where $\dim(U_i)=m$ and $n=mt$. As $Y$ acts absolutely irreducibly on $V$, $Y$ is not contained in the subgroup $M_{(\mathscr{D})} \trianglelefteq M$ fixing each component $U_i$ for $1 \leqslant i \leqslant t$, and so there exists a homomorphism from $Y$ into $M/M_{(\mathscr{D})} \leqslant \mathrm{S}_t$ such that the image of $Y$ acts transitively on the $t$ components. As $Y$ has no transitive action on $7,8$ or $9$ points we must have $n=t=6$. But this yields a contradiction, as the only insoluble composition factor of a group in $\mathscr{C}_2$ in this case is $\mathrm{A}_6$. 
\par Next suppose $M$ is contained in a member of $\mathscr{C}_4$ or $\mathscr{C}_7$, so that $M$ preserves a decomposition $\mathscr{D}$ of the form $V=U_1 \otimes \dots \otimes U_t$ where $\dim(U_i)=n_i$ for $1 \leqslant i \leqslant t$ and $n=n_1n_2 \dots n_t$. The insoluble composition factors of groups in these classes are classical groups with natural modules of dimension at most 3. It follows from \cite[Proposition~5.3.7]{KlLi} that $t=2$, and from the structure of $M$ we have $X$ and $Y$ contained in the direct product of two classical groups with natural modules of dimensions $(n_1, n_2)=(2,4)$ or $(3,3)$. In the first instance the projection of $X$ onto the first factor must be trivial, and so $X$ is contained in the second factor. But this implies $Y \leqslant X$ is reducible, a contradiction. In the second case we have $n=9$, and there exists no maximal subgroup $N \in \mathscr{C}_9$.
\par Finally consider the case where $M$ is contained in a member of $\mathscr{C}_6$, the normaliser of an extraspecial $r$-group for a prime $r$. Inspection of the tables in \cite{BHRD} shows that the only insoluble composition factors of such groups containing $Y$ are $\Sp_6(2)$ and $\mathrm{A}_8$ in the cases $T=\PSL^\epsilon_8(q)$ and $\POmega_8^+(q)$ respectively. However, neither of these cases has a maximal subgroup $N \in \mathscr{C}_9$. This competes the proof of (\ref{eqn:YinS}).

\par Using (\ref{eqn:YinS}), we use the tables in \cite{BHRD} and find that $T$ is listed above, taking care to inspect the tables for exceptional isomorphisms involving almost simple groups with socle $\mathrm{A_6}$, such as 
\begin{align*}
\mathrm{A_6} \cong \PSL_2(9) \cong \PSp_2(9) \cong \PSU_2(9) \cong \Omega_3(9) \cong \Sp_4(2)' \cong \Omega_4^-(3).
\end{align*}
\noindent If there exists a maximal subgroup $N$ with $N \notin \mathscr{C}_9$ (such as in the case where $T=\PSU_4(3)$ and $N \in \mathscr{C}_5$) then \textsc{Magma} is often useful for showing that $Y$ is not contained in an almost simple group with socle $\mathrm{A}_7$.

\par Now suppose $n \geqslant 10$. Here Aschbacher's theorem applies and so $N$ lies in one of $\mathscr{C}_1, \dots, \mathscr{C}_9$. Either $T=\PSL_n(q)$ and $G$ contains a graph automorphism or $G \leqslant \PGammaL_n(q)$, and so $G$ acts on the set of subspaces of $V$.
\par We first prove that
\begin{align} \label{eqn:MNinC9}
N \in \mathscr{C}_9.
\end{align}
\par We begin by noting that if $S$ is a finite simple classical group with natural module of dimension at least 5, then 
\begin{align} \label{eqn:MNorder}
|N| < |S|.
\end{align}
\par First suppose $N \in \mathscr{C}_1$. If $N$ stabilises a decomposition $V=U \oplus W$ with $\dim(U)=m, \dim(W)=n-m$ and $m \geqslant \frac{n}{2} \geqslant 5$ then by \cite[Lemma~4.1.1]{KlLi} we have $\overline{ \left( \Omega(U) \times \Omega(W) \right)} \leqslant  N \cap T $, contradicting (\ref{eqn:MNorder}). Therefore either $N$ stabilises a totally singular $m$-space, or $T=\PSL_n(q)$ and $N$ stabilises a pair of subspaces $\{U ,W \}$, or $T=\POmega^\pm_n(q), q$ is even and $N$ stabilises a nonsingular 1-space. In the first two cases, \cite[Propositions~4.1.17--4.1.22]{KlLi} imply that $N \cap T$ is $p$-local, a contradiction. In the remaining case we have $\soc(N \cap T) \cong \Sp_{n-2}(q)$ by \cite[Proposition~4.1.7]{KlLi}, again a contradiction. Hence $N \notin \mathscr{C}_1$. 
\par Next suppose that $N \in \mathscr{C}_2$, so that $N$ stabilises a decomposition $V=U_1 \oplus \dots \oplus U_t$ where $t \geqslant 2, \dim(U_i)=m$ for $1 \leqslant i \leqslant t$ and $n=mt$. By \cite[Propositions~4.2.4--4.2.7, 4.2.9--4.2.11, 4.2.14--4.2.16]{KlLi}, either $N \cap T$ has a normal subgroup of index $t!$ or $N$ is of type $O_\frac{n}{2}(q)^2$, and so we must have $t=2$. The same results also imply that $\overline{\Omega(U_1)} \leqslant N \cap T$, contradicting (\ref{eqn:MNorder}) since $m=\frac{n}{2} \geqslant 5$. Hence $N \notin \mathscr{C}_2$.
\par We now consider $N \in \mathscr{C}_3$, so that $N$ is of the form $\mathrm{Cl}_m(q^t)$ for some classical group and prime $t$ such that $n=mt$. By \cite[Propositions~4.3.i, i=6,7,10,14,16,17]{KlLi}, $N \cap T$ contains a normal subgroup of index $t$, and so $t=2$. However, the same results imply that $N \cap T$ contains a classical group $\mathrm{Cl}_\frac{n}{2}(q^2)$, contradicting (\ref{eqn:MNorder}). 
\par Next suppose $N \in \mathscr{C}_4$, stabilizing a tensor product decomposition $V=U \otimes W$ with $\dim(U)=n_1, \dim(W)=n_2$ and $n=n_1n_2$. However, since $N \cap T$ is non-local we have $\soc(N \cap T) \cong \POmega(U)' \times \POmega(W)'$ by \cite[Lemma~4.4.9]{KlLi}, a contradiction. Therefore $N \notin \mathscr{C}_4$.
\par Now suppose $N \in \mathscr{C}_5$, so that $N$ is of the form $\mathrm{Cl}_n(q^\frac{1}{t})$ for some classical group and some prime $t$. Since $N \cap T$ is non-local, by \cite[Proposition~4.5.2]{KlLi} we have $\soc(N \cap T) \cong S$ for some finite simple classical group $S$ with natural module of dimension $n$, contradicting (\ref{eqn:MNorder}). Hence $N \notin \mathscr{C}_5$.
\par We next suppose $N \in \mathscr{C}_6$. However, \cite[Propositions~4.6.5--4.6.9]{KlLi} imply that $N \cap T$ is $r$-local for some prime $r$, a contradiction. 
\par If $N \in \mathscr{C}_7$ then $N$ stabilises a tensor product decomposition $V=U_1 \otimes \dots \otimes U_t$ where $\dim(U_i)=m$ for $1 \leqslant i \leqslant t$ and $n=m^t$. However, \cite[Lemma~4.7.1]{KlLi} implies that $\soc(N \cap T) \cong \POmega(U_1)' \times \dots \times \POmega(U_t)'$, a contradiction. Hence $N \notin \mathscr{C}_7$.
\par Finally consider the case $N \in \mathscr{C}_8$. Since $N$ is non-local, \cite[Lemma~4.8.1]{KlLi} implies that $\soc(N \cap T) \cong S$ for some finite simple classical group $S$ with natural module of dimension $n$, a contradiction. This proves $N \notin \mathscr{C}_8$.
\par We therefore have $N \in \mathscr{C}_9$, and so we have proved (\ref{eqn:MNinC9}). 

\par From the description of $\mathscr{C}_9$ given in \cite[\S~1.2]{KlLi}, $Y$ (and therefore $X$) acts irreducibly on $V$, and so $\hat{X}$ and $\hat{Y}$ are irreducible. By \cite[33.3]{As} we have $\hat{X}=\hat{X}' \circ Z(\hat{X})$ and $\hat{Y}=\hat{Y}' \circ Z(\hat{Y})$, and so $\hat{X}'$ and $\hat{Y}'$ act irreducibly on $V$. As $\hat{X}'$ and $\hat{Y}'$ are perfect central extensions of $\mathrm{A_7}$ and $\mathrm{A_6}$ respectively, $\hat{X}' \cong \mathrm{A_7}, 2.\mathrm{A_7}, 3.\mathrm{A_7}$ or $6.\mathrm{A_7}$, and similarly $\hat{Y}'\cong \mathrm{A_6}, 2.\mathrm{A_6}, 3.\mathrm{A_6}$ or $6.\mathrm{A_6}$. Using \cite{Atlas} and \cite{Brauer}, we find that the degrees of the irreducible representations of $\hat{X}'$ and $\hat{Y}'$  with $n \geqslant 10$ only coincide for $n=10$ or $15$. 
\par Using \cite[Tables~8.60--8.69]{BHRD} for $n=10$ and \cite[Tables~11.0.17--11.0.22]{Sc} for $n=15$, a similar proof to (\ref{eqn:YinS}) shows that $M \in \mathscr{C}_9$. We find that the only possibility is $n=10$ and $T=\POmega_{10}^\epsilon(q)$ (noting the exceptional isomorphisms listed above) with $M \leqslant R \max G, R \in \mathscr{C}_9$ and $\soc(R)=\PSp_4(q), \mathrm{M}_{22}, \mathrm{A}_{11}$ or $\mathrm{A}_{12}$. By \cite[Theorem~4.3.3]{BHRD}, $\hat{Y}'$ only has an absolutely irreducible representation of degree 10 preserving a quadratic form in characteristic $p=7$ (and in this case there is a unique representation), and so from \cite[Tables~8.67, 8.69]{BHRD} we have $\epsilon=-$ and $q=7$. This rules out the case $\soc(R)=\mathrm{A}_{12}$. We can also discard the case $\soc(R)=\mathrm{A}_{11}$, as the 10-dimensional module for $\mathrm{A}_{11}$ is the deleted permutation module, and the restriction to $\mathrm{A}_7$ is reducible. Therefore $\soc(R)=\mathrm{M}_{22}$ or $\PSp_4(7)$. In fact, in both of these cases we have $Y$ contained in a subgroup of $\soc(R)$ isomorphic to $\mathrm{A}_7$ by \cite[Propositions~4.9.60, 4.9.63, 6.2.13]{BHRD}, and so we list $T$ above.
\end{proof}

\begin{proof}[Proof of Theorem \ref{thm:faithful}]
\par As usual we follow the method outlined in Section \ref{sec:prelim} and classify subgroup lattices $L=(G,E,A,H)$ up to $\Aut(G)$-conjugation, where $G$ is an almost simple classical group, $E \max G$, $A<E$ with $|E:A|=2$ and $A<H<G$ with $T=\soc(G) \not\leqslant H$. Moreover, by Lemma \ref{lem:lufaithful}, $E$ and $H$ are almost simple with $\soc(E)=\mathrm{A_6}, \soc(H)=\mathrm{A_7}$, $|H:A|=7$ and $E \ne \mathrm{A_6}$ or $\mathrm{S_6}$. Note that in order for $A$ to be a proper subgroup of $H$ we must have $A \cong \mathrm{A}_6$ or $\mathrm{S}_6$. We then consider the coset graph $\Gamma$ obtained from each such lattice $L$. Clearly $G \leqslant \Aut(\Gamma)$, and so to determine if $G$ is the full automorphism group of $\Gamma$ it suffices to check \cite{LPS}. By construction and Lemma \ref{lem:lufaithful}, $\Gamma$ is edge-primitive and 3-arc-transitive with automorphism group $G$, and we list the lattice $L$ in Table \ref{tab:classicalgroupsfaithful}.

\par Let $G$ be an almost simple classical group with socle $T$, and suppose $G$ contains subgroups $E, A$ and $H$ as above. Observe that $E$ and $H$ satisfy the hypotheses of Lemma \ref{lem:a6a7maximalclassical} with $E=N$ and $H=M$, and therefore $T=\PSU_3(5), \PSL_4(2), \PSL_6^\epsilon(q)$ or $\POmega_{10}^-(7)$. Since $\PSL_4(2) \cong \mathrm{A}_8$ and Theorem \ref{thm:Anspor} shows that there are no edge-primitive 3-arc-transitive graphs whose automorphism group has socle $A_8$, we have already eliminated the case where $T=\PSL_4(2)$.

\par First suppose $T=\PSU_3(5)$. By \cite[Tables~8.5--8.6]{BHRD} we have $T \leqslant G \leqslant T.2$, where $T.2$ is an extension by a graph automorphism, and there are unique $\Aut(T)$-classes of subgroups $E$ and $H$ of $G$ such that $\soc(E)=\mathrm{A_6}, \soc(H)=\mathrm{A_7}$ and $E \max G$ (here $E$ and $H \in \mathscr{C}_9$ and $H \max G$). Using \textsc{Magma} we find that, for both $G=T$ and $G=T.2$, there exists a unique subgroup $A<E$ of index 2 such that $A<H_0<G$ where $H_0$ is $\Aut(T)$-conjugate to $H$. Therefore we have a single lattice, and this yields an edge-primitive 3-arc-transitive graph $\Gamma$ with automorphism group containing $G=T.2$. By \cite{LPS}, $G$ is indeed the full automorphism group, and so we have a lattice as listed in Table \ref{tab:classicalgroupsfaithful}. The associated graph is the Hoffman-Singleton graph.

\par Next suppose $T=\POmega_{10}^-(7)$. By \cite[Tables~8.68, 8.69]{BHRD} we have $T.2 \leqslant G \leqslant T.2^2$ where $T.2^2=T.\langle \gamma, \delta' \rangle = \PO_{10}^-(q)$ and $T.2 \ne T.\langle \delta' \rangle$. There are unique $\Aut(T)$-class of subgroups $E$ and $H$ of $G$ such that $\soc(E)=\mathrm{A_6}, \soc(H)=\mathrm{A}_7, E \max G$ and $\soc(E)<H$   (here $E \in \mathscr{C}_9$, and the result for $H$ follows from \cite[Theorem~4.3.3]{BHRD}). Since $\left( E \cap T\right).2\max T.2$ the graph is $T.2$-edge-primitive. We have $\left( E \cap T \right).2 \cong \PGL_2(9)$ or $\mathrm{M}_{10}$, and from the final paragraph of the proof of Lemma \ref{lem:a6a7maximalclassical},  we have $A \cap T  = \soc(E)$ contained in a subgroup $H_0 \cong \mathrm{A}_7$ that is $\Aut(T)$-conjugate to $H \cap T$. From \cite[Propositions~4.9.60, 6.2.13]{BHRD}, $H_0 < \mathrm{M}_{22} \cap \PSp_4(7)$ where $\mathrm{M}_{22} \max T$ and $\PSp_4(7) \max T$. We therefore have a single lattice, and this yields an edge-primitive 3-arc-transitive graph $\Gamma$ with $T.2 \leqslant \Aut(\Gamma)$. The lattice extends to $G=T.2^2$ with $E=\Aut(\mathrm{A}_6), A=\mathrm{S}_6 = \left( E \cap T \right).\langle \delta' \rangle$ and $H=\mathrm{S}_7 \leqslant \PSp_4(7).2 \max T.\langle \delta' \rangle$, and by \cite{LPS} this is the full automorphism group of $\Gamma$. 

\par Finally we consider the case where $T=\PSL_6^\epsilon(q)$. This case is more involved, as by \cite[Tables~8.24--8.27]{BHRD} we see that there are infinitely many values of $q$ such that $G$ contains almost simple subgroups $E$ and $H$ with $\soc(E)=\mathrm{A_6}, \soc(H)=\mathrm{A_7}$ and $E \max G$. First suppose $\epsilon=+$.  We shall work in the quasisimple group with $T=\SL_6(q)$. There are at most two $\Aut(G)$-classes of subgroups $E \max G$ with $\soc(E)=\mathrm{A_6}$, depending on if $q=p$ or $q=p^2$ and the value of $q \pmod{48}$. For each class, we consider if $E$ has a subgroup of index 2 contained in some almost simple group $H$ as above. 
\par We first note that, from \cite[Theorem~4.3.3]{BHRD}, absolutely irreducible representations of quasisimple groups of type $\mathrm{A}_7$ in dimension 3 only occur in characteristic 5. Therefore, if we have $E, A$ and $H$ as above (in the quasisimple group) and $H \leqslant (q-1, 6) \circ \SL_3(q) \in \mathscr{C}_9$ then by the conditions on the maximality of $E$ we have $q=25$ and $E \cap T = 6^\cdot \mathrm{A}_6$. However, a \textsc{Magma} calculation reveals that in $\SL_6(25)$ the subgroup $6\circ\SL_3(25)$ does not have a subgroup isomorphic to $6^\cdot \mathrm{A}_6$, and so this yields a contradiction. Therefore $H \not\leqslant (q-1, 6) \circ\SL_3(q)$.

\par First suppose $E \cap T=2 \times 3^\cdot \mathrm{A}_6.2_3$ is of novelty type N1 as in \cite[Table~8.25]{BHRD} so that $q=p \equiv 1 \pmod{24}$. Then $G=T.2$, an extension of $G$ by a graph automorphism, $E=2 \times 3^\cdot \mathrm{A}_6.2^2$ and $A=2 \times 3^\cdot \mathrm{A}_6.2_1=2 \times 3^\cdot \mathrm{S}_6$. Therefore $A \not\leqslant T$ and so $H \not\leqslant T$. Therefore the only possible maximal overgroups of $H$ are $6_1.\PSU_4(3).2_1$. However, from \cite{Atlas} we have $\mathrm{S}_7 \not\leqslant \PSU_4(3).2_1$, a contradiction. Therefore this case does not give rise to a lattice.

\par The proof where $E \cap T =2 \times 3^\cdot \mathrm{A}_6.2_3$ is a novelty of type N2 is very similar, and so we now consider the case where $E \cap T = 2 \times 3^\cdot \mathrm{A}_6$ is of type N3. In this case $q = p \equiv 7 \pmod{24}$, and either $G=T.\langle \delta^3, \gamma \rangle$ or $p \equiv \pm 2 \pmod{5}$ and $G=T.\langle \gamma \rangle$. First suppose $G=T.\langle \delta^3, \gamma \rangle$, so that $E = 2 \times 3^\cdot \Aut(\mathrm{A_6})$ and $A = 2 \times 3^\cdot \mathrm{S_6} = \left( E \cap T\right).\langle \delta^3 \gamma \rangle$. The only possibility is $A < H \leqslant 6_1 ^\cdot \PSU_4(3).2^2 \max G$, and indeed this chain exists (with $H \cong 2 \times 3 ^\cdot \mathrm{S}_7$) by \cite[Propositions~4.8.9, 4.8.12]{BHRD}. There is a unique $\Aut(T)$-class of absolutely irreducible representations of $3^\cdot \mathrm{A}_7$ by \cite[Theorem~4.3.3]{BHRD}, and so this yields a single lattice which we add to our table. This gives rise to an edge-primitive 3-arc-transitive graph $\Gamma$ with automorphism group $G$, and $\Gamma$ is the spread of an edge-primitive vertex-primitive graph with vertex-stabiliser $\PSU_4(3).2^2$ and the same automorphism group and edge-stabiliser. The case with $p \equiv \pm 2 \pmod{5}$ gives the same result.

\par The case where $E$ is novelty of type N4 is dealt with in a similar manner, and so we now consider the case where $E \cap T = 6^\cdot \mathrm{A}_6$ is of novelty type N1. Here $q=p \equiv 1, 31 \pmod{48}$ and $G=T.2$ is an extension of $T$ by a graph automorphism. We have $E=6^\cdot \mathrm{A}_6.2_2$ and $A=6^\cdot \mathrm{A}_6=E \cap T$. From \cite[Proposition~4.8.8]{BHRD}, $A<6^\cdot \mathrm{A}_7 \max T$, and this is the only possibility for $H$. We therefore have a lattice as desired, yielding a coset graph $\Gamma$. In this case $G$ is the full automorphism group of $\Gamma$ by \cite{LPS}. We therefore list the lattice in Table~\ref{tab:classicalgroupsfaithful}. 

\par The proofs for the remaining $\Aut(T)$-classes of subgroups $E$ of $G$ are very similar (using \cite[Propositions~4.8.8, 4.8.12]{BHRD}), and the results are listed in Table~\ref{tab:classicalgroupsfaithful}. 
\par The case $T= \PSU_6(q)$ is almost identical to the proof in the linear case.
\end{proof}

\bibliographystyle{amsplain}

\end{document}